\documentclass[a4paper,11pt]{article}

\usepackage[english]{babel}
\usepackage[utf8]{inputenc}
\usepackage{paralist,url,verbatim,anysize}
\usepackage{amscd}
\usepackage{color}
\usepackage{amsmath,amsthm,amssymb,amsfonts}
\usepackage{graphicx, graphics}
\usepackage{hyperref}
\usepackage{epstopdf}
\theoremstyle{plain}
\newtheorem{theorem}{\bf Theorem}[section]

\newtheorem{cor}[theorem]{Corollary}
\newtheorem{lemma}[theorem]{Lemma}
\newtheorem{prp}[theorem]{Proposition}

\newtheorem{problem}[]{Problem}

\newcounter{main}

\newtheorem{thmnonumber}[main]{\bf Main Theorem}

\newtheorem{prpnonumber}[main]{\bf Proposition}

\theoremstyle{definition}
\newtheorem{rem}[theorem]{Remark}
\newtheorem{definition}[theorem]{Definition}

\newtheorem{example}[theorem]{Example}

\newtheorem{conjecture}[theorem]{Conjecture}

\DeclareMathAlphabet{\mathpzc}{OT1}{pzc}{m}{it}
\newenvironment{frcseries}{\fontfamily{frc}\selectfont}{}

\newcommand{\II}{\mathbb{I}}

\newcommand{\R}{\mathbb{R}}
\def \cd{\operatorname{cdepth}}
\newcommand{\Lk}{\mathrm{Link} \,}

\newcommand{\sd}{\mathrm{sd}}

\newcommand{\Ziv}{\v{Z}ivaljevi\`{c} }

\begin{document}

\author{
Bruno Benedetti \thanks{The author was supported by: the DFG Collaborative Research Center TRR109, ``Discretization in Geometry and Dynamics''; the G\"{o}ran Gustaffsson Foundation, Stockholm, Sweden; the Swedish Research Council, grant ``Triangulerade M{\aa}ngfalder, Knutteori i diskrete Morseteori''; the European Research Council, under the European Union's seventh framework programme (FP7/2007-2013) / ERC Grant agreement no.~247029-SDModels.} \\
\small
Dept.\ Mathematik und Informatik, Freie Univ. Berlin\\
\small \url{bruno@zedat.fu-berlin.de}}

\date{\small December 10, 2014}
\title{Smoothing discrete Morse theory}
\maketitle

\begin{abstract}
After surveying classical notions of PL topology of the Seventies, we clarify the relation between Morse theory and its discretization by Forman. We show that PL handles theory and discrete Morse theory are equivalent, in the sense that every discrete Morse vector on some PL triangulation is also a PL handle vector, and conversely, every PL handle vector is also a discrete Morse vector on some PL triangulation. 
It follows that in dimension up~to~$7$, every discrete Morse vector on some PL triangulation is also a smooth Morse vector; the vice versa is true in all dimensions.  This revises and improves a result by Gallais.

Some further consequences of our work are:
\begin{compactenum}[(1)]
\item For $d \ne 4$, every simply connected smooth $d$-manifold admits locally constructible triangulations. In contrast, the Mazur $4$-manifold has no locally constructible triangulation. (This solves a question by \Ziv  and completes work by the author and Ziegler.)
\item The Heegaard genus of $3$-manifolds can be characterized as the smallest $g$ 
for which some triangulation of the manifold has discrete Morse vector $(1,g,g,1)$.\\
(This allows for heuristics to bound the Heegaard genus of any 3-manifold.)
\item Some non-PL $5$-spheres admit discrete Morse functions with only 2 critical faces. \\
(This result, joint with Adiprasito, completes the Sphere Theorem by Forman.)
\end{compactenum} 
\end{abstract}

\enlargethispage{4mm}
{\small \tableofcontents}

\section*{Introduction}
The term ``$d$-dim\-ensional manifold'' was introduced by Riemann in 1854 \cite{Ratcliffe, Riemann}. Since then, the notion has been used and adapted so many times, that it is rare to find two books adopting the same definition. Here by \emph{topological manifolds} we mean ``compact Hausdorff spaces locally homeomorphic to $\mathbb{R}^d$'', or, in the case of manifolds with boundary, to a closed halfspace of $\mathbb{R}^d$.\footnote{Other authors adopt weaker definitions, e.g.\ by removing the Hausdorff requirement, or by dropping the compactness assumption in favour of paracompactness. }
Many topological manifolds carry also other important structures. 
For example, \emph{smooth manifolds} are locally diffeomorphic to $\mathbb{R}^d$, which allows us to do differential calculus on them. Finally, \emph{triangulable manifolds} are homeomorphic to simplicial complexes, which allows us to tackle them algorithmically. 
These three structures (topological, differential, combinatorial) are inequivalent: smooth manifolds are only some of the triangulable manifolds, which in turn are only some of the topological manifolds.

In the Thirties, Morse had the idea to analyze a  smooth manifold by considering a generic smooth real-valued function defined on it. 
The highlight of the theory is that it suffices to look at the finitely many critical points (maxima, minima, saddles) of the function, to understand the homotopy type of the whole manifold. 

\begin{theorem}[Morse~\cite{Morse}] \label{thm:morse}
If some $d$-dimensional smooth manifold $M$ admits a generic smooth function $f:M\rightarrow\R$ that has $c_i$ critical points of index $i$ ($i=0,\ldots, d$), then $M$ is homotopy equivalent to a cell complex (called \emph{Morse complex}) with $c_i$ cells of dimension $i$.
\end{theorem}

From now on, the assumption above will be shortened as ``$M$ has \emph{Morse vector} 
$\mathbf{c}=(c_0, \ldots, c_d)$''. The Morse vector in general does not determine the Morse complex nor the homotopy type of $M$, but it yields useful information: using cellular homology, one can see that the Morse vector is always coordinate-wise greater than or equal to the \emph{Betti vector} $\mathbf{\beta} = (\, \beta_0(M),  \ldots, \beta_d (M) \,)$. 

In 2000, Forman published a combinatorial version of Morse theory, which applies to arbitrary simplicial complexes --- hence in particular to triangulations of manifolds. Simplifying a bit Forman's definition, a discrete Morse function on a simplicial complex $C$ is a map $f : C \rightarrow \mathbb{Q}$ that satisfies three properties: 
\begin{compactenum}[(i)]
\item (Monotonicity) If $\sigma \subset \tau$, then $f(\sigma) \le f(\tau)$.
\item (Semi-injectivity) For each $q \in \mathbb{Q}$, the cardinality of $f^{-1} (q)$ is at most $2$.
\item (Genericity) If  $f(\sigma) = f(\tau)$, then either $\sigma \subset \tau$ or $\tau \subset \sigma$.
\end{compactenum}
If one defines \emph{critical faces} as the faces $\sigma$  at which $f$ is injective (i.e.\ the cardinality of $f^{-1} (f(\sigma))$ is equal to $1$), Forman discovered that it suffices to look at critical faces to understand the starting complex up to homotopy, exactly like in Morse theory:

\begin{theorem} [Forman~\cite{FormanUSER}]
If some $d$-dimensional simplicial complex $C$ admits a discrete Morse function $f:C\rightarrow\mathbb{Q}$ that has $c_i$ critical faces of dimension $i$ ($i=0,\ldots, d$), then $C$ is homotopy equivalent to a cell complex with $c_i$ cells of dimension $i$.
\end{theorem}

Let us shorten the assumption above as ``$C$ has \emph{discrete Morse vector} $\mathbf{c}=(c_0, \ldots, c_d)$''; as before, the discrete Morse vector does not determine the homotopy type of $C$, but it bounds coordinate-wise the Betti vector of $C$ from above. This naturally leads to the following problem.

\begin{problem}
What is the relation between Morse vectors and discrete Morse vectors on a fixed topological manifold?
\end{problem}

A first important contribution to this problem came in 2010 from Gallais:

\begin{theorem}[Gallais \cite{Gallais}]
If  a smooth manifold $M$ admits a Morse vector $\mathbf{c}$, then there exists a PL triangulation $T$ of $M$ that admits $\mathbf{c}$ as discrete Morse vector. 
\end{theorem} 

The proof in \cite{Gallais} (which aims at discretizing the triangulation alongside the gradient flows) contains a minor gap, which we are able to fix in this paper (cf.\ Remark \ref{rem:gallaisgap}). In fact, here we provide an alternative proof of the following, stronger fact:

\begin{thmnonumber} [Theorem \ref{thm:PLtoDiscrete} and Remark \ref{rem:StoPLhandle}]\label{mthm:GallaisImproved}
If  a smooth manifold $M$ admits a Morse vector $\mathbf{c}$, then for \emph{any} PL triangulation $T$ of $M$ one can find an integer $r$ such that the $r$-th derived subdivision of $T$ admits $\mathbf{c}$ as discrete Morse vector. 
\end{thmnonumber}

Our proof of Main Theorem \ref{mthm:GallaisImproved} is based on three ingredients:

\begin{compactenum}[(1)] 
\item handle decompositions, a classical topological tool to chop the $d$-manifold into $d$-balls (called ``handles'') with controlled intersection; see e.g.\ Rourke--Sanderson~\cite{RourkeSanderson}; 
\item the shelling and collapsing techniques for $d$-balls established in Adiprasito's PhD thesis \cite{Adiprasito},  (cf. also \cite{AB-SSZ} and \cite[Main Theorem VI]{AB-MG&C}); these tools allow us to simultaneously triangulate all handles nicely, with a suitable (iterated) barycentric subdivision;
\item the duality for discrete Morse functions on manifolds with boundary, and the theory of composing together (dual) discrete Morse functions, as developed by the author in \cite{B-DMT4MWB}.
\end{compactenum}

\smallskip
By looking at (differences of) sublevel sets, every smooth Morse function provides automatically a handle decomposition, with each handle containing exactly one of the critical points in its interior. In 2009, Jer\u{s}e and Mramor discovered that discrete Morse theory yields a similar way to decompose (part of) a triangulation $C$ into ``descending regions'', each containing one critical face \cite{JerseMramorKosta}. There are however two main differences between this decomposition and the classical handle decomposition:
\begin{compactenum}[(i)]
\item The  Jer\u{s}e--Mramor descending regions are not necessarily full-dimensional. (This can be fixed, however, by passing to regular neighborhoods.) 
\item The union of the neighborhoods of the descending regions is not the whole complex $C$ (unless $C$ is a manifold with empty boundary \cite[Prop. 2]{JerseMramorKosta}), but rather a deformation retract of it. This retract need not be homeomorphic to $C$, as highlighted in Remark~\ref{rem:NonHomeomorphic}.
\end{compactenum}

\noindent The example in Remark \ref{rem:NonHomeomorphic}, however, is not a PL triangulation. Here PL means Piecewise Linear, and refers to the technical assumption that each vertex star must have a subdivision isomorphic to  some subdivision of the simplex. 

Our new result is that if $C$ is a PL triangulation, the union $U$ of the regular neighborhoods of the descending regions is indeed homeomorphic to $C$ --- and a discrete Morse function on the descending regions can be extended to $C$ without adding further critical cells. This was first conjectured by Bauer and Mramor-Kosta, in private communications to the author, for which we wish to thank them. The key technical tool is once again the shelling technique of Theorem \ref{thm:subdivision}, cf. \cite[Theorem 5.3]{AB-SSZ}, which allows us to strengthen Whitehead's Regular Neighborhood theorem (Corollary \ref{cor:shellbary}).

Combining this with the known relations between smooth Morse vectors and PL handle vectors (see Section \ref{sec:PLdiscrete} for the definition), we obtain the following new result, converse to Main Theorem \ref{mthm:GallaisImproved} and reflecting joint work with Adiprasito:

\begin{thmnonumber} [Theorems \ref{thm:smoothtoPL} and \ref{thm:dmttopl}]\label{mthm:A}
For manifolds admitting PL triangulations, every PL handle vector is also a discrete Morse vector on some PL triangulation, and the other way around.
Moreover, if the manifold is smooth and has dimension $\le 7$, then every discrete Morse vector on some PL triangulation is also a smooth Morse vector.
\end{thmnonumber}

In conclusion, if we restrict ourselves to PL triangulations and to dimensions up to $7$, smooth and discrete Morse theory are \emph{equivalent} when it comes to bounding Betti numbers. 

Using duality, we can provide an application to  {\em locally constructible} (or shortly, \emph{LC}) triangulations, introduced in \cite{DJ} and studied in \cite{BZ}. These are  triangulations of simply connected manifolds that enjoy a special inductive combinatorial structure; see Section \ref{sec:LC} for the precise definition.
In 2009 \Ziv conjectured that every simply connected smooth manifold admits some LC triangulations (personal communication). Here we are able to confirm the conjecture in all dimensions except $4$, in which it is false:

\begin{thmnonumber} [Theorem \ref{thm:zv}] \label{mthm:ziv}
Let $d \ge 2$. Every simply connected smooth $d$-manifold, $d \ne 4$, admits LC triangulations. (In fact, any PL triangulation of such manifold becomes LC after subdividing it barycentrically suitably many times.) In contrast, the Mazur $4$-manifold admits no LC triangulation.
\end{thmnonumber}

The proof makes use of Perelman's solution of the $3$-dimensional Poincar\'e conjecture. In fact, for $d=3$ the statement above \emph{is} a combinatorial reformulation of the Perelman--Poincar\'e theorem, since Durhuus--Jonsson showed with elementary methods that all LC closed $3$-manifolds are spheres (Theorem \ref{thm:DJ}).

Let us spend a few more words on the case $d=3$. All triangulations of $3$-manifolds are PL. Hence, Main Theorems \ref{mthm:GallaisImproved} and \ref{mthm:A} imply that a closed $3$-manifold $M$ admits a Heegaard splitting of genus $g$ if and only if some triangulation of $M$ has discrete Morse vector $(1,g,g,1)$. This allows us to re-write in the language of discrete Morse theory a recent result by Li \cite{TaoLi}:  For every $g>0$, there is a  $3$-manifold $M$ such that any discrete Morse function on any triangulation of $M$ has more than $g + \operatorname{rank} (M)$ critical edges (Corollary \ref{cor:li}).  
This shows a topological obstruction to the existence of nice discrete Morse functions. Unlike the knot-theoretic obstruction discussed in \cite{B-DMT4MWB}, this obstruction cannot be removed by performing convenient subdivisions. Moreover, it allows for quick experimental methods to understand the Heegaard genus, or at least to bound it from above.

Finally, we turn to an issue left open by Forman's Sphere Theorem, which claims that every closed manifold with discrete Morse vector $(1, 0, \ldots, 0, 1)$ is homeomorphic to the $d$-sphere. Forman's proof distinguishes two cases: the PL case, where the conclusion follows from a result by Whitehead, and the non-PL case, in which one has to use the Poincar\'e conjecture.  However, due to the lack of examples, it is unclear whether the non-PL case is void or not. Here using a result by Adiprasito, we show it is not:

\begin{prpnonumber}[Proposition \ref{prp:forman}]
\label{mthm:forman}
For each $d\geq 5$, some non-PL triangulation of the $d$-sphere admits a discrete Morse function with only two critical faces. 
\end{prpnonumber}

In particular, there exists a non-PL $5$-ball $B$ that is endo-collapsible, that is, a $5$-ball $B$ that after the removal of any facet collapses onto $\partial B$. 
In \cite{B-DMT4MWB} we proved (using Smale's solution of the Poincar\'e conjecture) that all endocollapsible manifolds are balls. This statement implies Whitehead's theorem (``all collapsible PL manifolds are balls'') and in view of the $5$-ball $B$ above, the implication is strict.

\bigskip

\emph{\textbf{Structure}}. The paper is organized as follows. For the sake of clarity and self-containment, we decided to start with a background chapter gathering together the differential and PL topology we need. These notions can also be found in many textbooks (for example \cite{MilnorFifty}, \cite{RourkeSanderson}, \cite{Buoncristiano}, \cite{Matsumoto}), though the notation vastly differs from book to book, which often creates confusion in the notation. The reader already familiar with PL triangulations and handle decompositions should skip directly to Section \ref{sec:PLdiscrete}.

\section{Background}
In this section, we recall a few classical results on how to triangulate manifolds and decompose them into handles, either in the smooth or in the PL sense.

\subsection*{Notation}

A \emph{topological $d$-manifold} (with boundary) is a compact Hausdorff space, locally homeomorphic to $\R^d$ (or to a halfspace in $\R^d$). For simplicity, all the manifolds we consider here are connected and orientable. A \emph{chart} on a topological $d$-manifold $M$ is a pair $(U,\phi)$, where $U$ is an open subset of $M$, and $\phi$ is some homeomorphism that maps $U$ into an open subset $\phi(U)$ of $\R^d$. An \emph{atlas} is a collection of charts $(U_i,\phi_i)$ such that the $U_i$ are open sets that cover $M$, and the induced transition maps between any two charts are smooth (that is, $C^{\infty}$). 
A  \emph{smooth $d$-manifold} is a topological $d$-manifold that admits an atlas. (Different atlas may give rise to the same manifold, so one usually requires the atlas to be maximal; see \cite[Lemma 1.10]{Lee} for details.)

A \emph{(simplicial) complex} is a non-empty, finite family $C$ of simplices in some $\R^k$ such that any two of them intersect at a common face, and every face of every simplex in $C$ is also an element of $C$. The \emph{dimension} of a complex is the maximum dimension of a simplex in it; \emph{$d$-complex} is short for ``$d$-dimensional simplicial complex''. A $d$-complex is \emph{pure} if all inclusion-maximal simplices in it have dimension $d$. The \emph{face poset} of a complex is the set of all of its simplices, ordered by inclusion. Two complexes are called \emph{combinatorially equivalent} if their face posets are isomorphic as posets. 
The \emph{underlying space} $|C|$ of a simplicial complex $C$ is the topological space given by the union of the simplices of $C$. If $X$ is any topological space homeomorphic to $|C|$, the complex $C$ is called a \emph{triangulation} of $X$. A \emph{subdivision} of a simplicial complex $C$ is a complex $C'$ with $|C|=|C'|$, such that for every face $F'$ of $C'$ there is a face $F$ of $C$ with $F' \subset F$. 

Two complexes are called \emph{PL-homeomorphic} if some subdivision of one is combinatorially equivalent to some subdivision of the other. Two triangulations that are PL homeomorphic can always be connected by a finite sequence of bistellar flips: See Pachner \cite{Pachner0}.
If $F$ is a simplex of a simplicial complex $C$, the \emph{star of $F$ in $C$} is the smallest subcomplex of $C$ that contains all faces of $C$ containing $F$.  
A \emph{PL ball} (resp.\ a \emph{PL sphere}) is a complex PL-homeomorphic to a simplex (resp.\ to the boundary of a simplex). A \emph{PL triangulation} of a given manifold is a triangulation in which the star of every face is a PL ball.

\subsection{The Poincar\'e conjecture}

A \emph{homology sphere} is a topological $d$-manifold with the same homology of the $d$-sphere. Not all homology spheres are spheres: Poincar\'e constructed homology $3$-spheres that are not simply connected, cf.\ \cite{BjornerLutz}. However he conjectured that all simply connected homology spheres must be spheres. After a century of progress, the conjecture proved right.

\begin{theorem}[Perelman ($d =3$) \cite{Perelman}, Freedman ($d =4$) \cite{Freedman}, Smale ($d \ge 5$) \cite{Smale}] \label{thm:poincaretop} \par
Every simply connected homology $d$-sphere is homeomorphic to the unit sphere in $\R^{d+1}$. 
\end{theorem}

\begin{cor}\label{cor:pwb}
Every contractible $d$-manifold with simply connected boundary is homeomorphic to the unit ball in $\R^{d}$. 
\end{cor}

It is natural to ask whether for a smooth manifolds, as in the previous two statements, ``homeomorphic'' can be replaced by ``diffeomorphic''. The answer is 
\begin{compactitem}
\item positive in dimension $d \in \{1,2,3,5,6,12,61\}$; 
\item unknown for $d=4$;
\item negative if $d \le 64$ and $d \notin \{1,2,3,4,5,6,12,61\}$ \cite{Milnor, MilnorKervaire}; see the very recent work by Behrens et al.\ for further cases \cite{BHHM}, e.g.\ for $64 \le d \le 126$;
\item unknown for large even values of $d$.
\end{compactitem}

However, in Theorem \ref{thm:poincaretop} and in Corollary \ref{cor:pwb} 
the conclusion ``homeomorphic'' can indeed be strengthened to ``PL-homeomorphic'', except possibly in dimension four:

\begin{theorem}[Perelman \cite{Perelman}, Smale \cite{Smale}, Stallings \cite{Stallings}, Zeeman \cite{ZeemanPoincare}] 
\label{thm:poincare} \par
Every simply connected PL homology $d$-sphere, if $d\neq 4$, is PL-homeomorphic to the boundary of the $(d+1)$-simplex.
\end{theorem}

\begin{cor}
Every PL contractible manifold of dimension $d\neq 4$ and with simply connected boundary is a \emph{PL ball}.
\end{cor}

In dimension $d=4$, the situation is quite unclear: See Section \ref{sec:4m}.

\subsection{All smooth manifolds can be PL triangulated} \label{sec:Smooth2PL}
To show that all smooth manifolds can be triangulated there are several ways; the argument we present here is due to Cairns \cite{Cairns}.
Without loss of generality, we may focus on smooth submanifolds of Euclidean space: In fact, by Whitney's embedding theorem, every smooth $d$-manifold 
is diffeomorphic to some submanifold of $\R^{2d}$.
If $B({\bf x},\varepsilon)$ denotes the open ball of radius $\varepsilon$ and center ${\bf x}$ in $\R^{2d}$, then obviously 
\[ \bigcup_{x \in M} B({\bf x}, \varepsilon) \; \supset \; M.\]
Since $M$ is compact, finitely many of these balls actually suffice to cover $M$. Any finite set of points ${\bf x}_1, \ldots, {\bf x}_s$ such that the union of the balls $B({\bf x}_i, \varepsilon)$ contains $M$ is called an \emph{$\varepsilon$-net of~$M$}.

\begin{theorem}[Cairns \cite{Cairns}]\label{thm:pltr}
Every smooth manifold can be PL triangulated.
\end{theorem}

\begin{proof}[Sketch of proof]
Let $\{{\bf x}_1, \ldots, {\bf x}_s\}$ be an $\varepsilon$-net of $M$. Consider now the Voronoi diagram $\mathfrak{V} \in \R^{2d}$ associated to these $s$ points. For each $i$, let us denote the cell containing ${\bf x}_i$ by $Z_i$. Consider the sets $C_i:=Z_i\cap M$. Since the $Z_i$ partition $\R^{2d}$, we have
$C_1 \cup \ldots \cup C_s = M$. Cairns showed that if $\varepsilon$  is small enough the $C_i$'s are cells, and the intersections $C_i\cap C_j$ are also cells. Hence, the $C_i$ form a regular cell complex. In general, a regular cell complex need not be a simplicial complex, but its barycentric subdivision is.
\end{proof}

The triangulation constructed in Theorem \ref{thm:pltr} is called a {\em Whitehead triangulation} of $M$. 
Up to PL-equivalence, it is the unique PL structure that satisfies  the so-called {\em Whitehead compatibility condition}, cf.~\cite{WhiteheadC1, KirSie} for details. In particular, any two PL triangulations obtained via Theorem \ref{thm:pltr} (with different choices of $\varepsilon$-nets) are PL-homeomorphic. In other words, up to PL equivalence it makes sense to speak of {\em the} Whitehead triangulation of $M$.

\subsection{Not all PL triangulations come from smooth structures} \label{sec:PLtoSmooth}

A smooth structure on a topological manifold $X$ is an equivalence class of smooth manifolds $M$ homeomorphic to $X$ that are related by orientation-preserving diffeomorphisms. We saw that from a smooth structure one can always pass to a PL structure. How about the converse? Can we pass from a PL structure to a smooth structure? The answer, in general, is negative.

Intuitively, the problem is this: Two triangulated $d$-manifolds $A$ and $B$ (with non-empty boundary) can always be glued together. For this  we just have to specify two combinatorially equivalent $(d-1)$-submanifolds $S_A \subset \partial A$ and $S_B \subset \partial B$, and then identify $S_A \equiv S_B$. If $A$ and $B$ are PL triangulations, their union will also be PL. In contrast, if we take two smooth $d$-manifolds $A$ and $B$, and we identify diffeomorphic $(d-1)$-submanifolds $S_A$ and $S_B$ of their boundaries, we might not obtain a smooth structure on the union. In fact, we should explain how to smoothly extend the attaching diffeomorphism (of $S_A$ and $S_B$) into (part of) the interiors of $A$ and of $B$.  
This is not always possible, as demonstrated by the following non-trivial counterexample, due to Kervaire \cite{Kervaire}.

\begin{example}[PL manifolds with no smooth structure]\label{ex:Kervaire}
Kervaire's $10$-dimensional topological manifold admits a PL triangulation. However, this triangulation is not homeomorphic to, and not even homotopy equivalent to, any smooth manifold \cite{Kervaire}. 
Munkres and Hirsch developed an obstruction theory to the existence of a smooth structures compatible with a given PL manifold \cite{Munkressmooth, HundM}. 
\end{example}

But supposing a smooth structure exists, can one recover it from a given PL structure? The answer in this case is also negative: Even if a compatible smooth structure exists, it need not be unique.

\begin{example}[PL spheres with many compatible smooth structures]\label{ex:manyforone}
Milnor constructed a smooth $7$-manifold homeomorphic to, but not diffeomorphic to, the unit sphere $S^7$ in $\R^8$ \cite{Milnorexotic}. The usual terminology is to speak of {\em exotic smooth structures} on the $7$-sphere. A single exotic structure produces many other exotic structures, just by taking connected sums. (It is a deep result in topology that smooth structures on $S^d$ form a finite Abelian group with respect to the connected sum \cite{MilnorKervaire}, except possibly for $d=4$.)
By Theorem \ref{thm:poincare}, all spheres of dimension different than $4$ admit a unique PL structure. In other words, if we apply the construction of Theorem \ref{thm:pltr} to all of Milnor's exotic $7$-spheres, we get PL triangulations that are PL-homeomorphic (and can be connected to each other via a finite sequence of bistellar flips), even if they come from different smooth structures.
\end{example}

However, all these problems occur in higher dimensions. If the dimension of the manifold is sufficiently small, any PL triangulation induces a unique smooth structure.

\begin{theorem}[{Kervaire--Milnor\cite{MilnorKervaire}, Hirsch--Mazur \cite{HundM}, Munkres \cite{Munkressmooth}}] \label{thm:UniqueSmoothStructure}
If $d \le 7$, every PL $d$-manifold has at least one compatible smooth structure; and if $d \le 6$, the compatible structure is unique up to diffeomorphism. 
\end{theorem}

We stress that the result above does \emph{not} say that $6$-manifolds have a unique smooth structure. It says  that the number of (non-diffeomorphic) smooth structures, for $d\le 6$,  is the same as the number of (non-PL-homeomorphic) PL triangulations. In dimension $d \in \{1,2,3\}$, this number is one, by the work of Papakyriakopoulos and Moise \cite{Moise}. However, if $d \in \{4,5,6\}$, several PL structures may appear. For example, Hsiang--Shaneson proved that the $5$-torus admits many smooth structures \cite{HsiangShaneson}.

\enlargethispage{3mm}

\subsection{Lack of (PL) triangulations}
In Section \ref{sec:Smooth2PL}, we reviewed Cairns' result that every smooth manifold admits a PL triangulation. This raises two natural questions: (1) Does every topological manifold admit a triangulation? (2) Is every triangulation PL?  

The answer to both questions is negative, as we will see. Let us start with a positive result:

\begin{theorem}[{Moise ($d \le 3$), Perelman ($d=4$)}] \label{thm:Perelman}
For $d \le 4$ all triangulations of any $d$-manifold are PL.
\end{theorem}

In particular, a $4$-dimensional topological manifold admits a triangulation if and only if it admits a PL structure; by Theorem \ref{thm:UniqueSmoothStructure}, this is equivalent to the existence of a smooth structure.

\begin{example}[Freedman \cite{Freedman}] 
Using a theorem by Rokhlin \cite{Rokhlin}, Freedman constructed a $4$-dimen\-sional topological manifold, the so called \emph{$E_8$-manifold}, that does not admit any smooth structure. It follows from Theorem \ref{thm:Perelman} (or alternatively, from Casson's work, cf. \cite{MilnorFifty}) that $E_8$ does not admit a triangulation either.
\end{example}

Whether all topological manifolds of dimension $> 4$ admit a triangulation was a long-standing open problem, 
recently solved by Manolescu, after it had been reduced to a question  on $3$-dimensional homology spheres (compare Galewski--Stern \cite{GalewskiStern} and Ranicki \cite{RanickiHV}):

\begin{theorem}[Manolescu \cite{Manolescu}]
For each $d \ge 5$ there exist topological $d$-manifolds that cannot be triangulated.
\end{theorem}

In addition, Theorem \ref{thm:Perelman} does not extend to higher dimensions: In the Seventies Edwards and Cannon discovered that some triangulations are not PL, as we will now explain. 

Recall that a \emph{homology sphere} is a manifold with the same homology as a sphere. Homology spheres need not be simply-connected. A \emph{homology-manifold} is a pure simplicial complex in which, for every face $\tau$, the $(\dim C - \dim \tau-1)$-complex $\Lk(\tau,C)$ has the same homology of a $(\dim C - \dim \tau-1)$-sphere. Every triangulation of a manifold is a homology manifold. The converse is false: Some homology manifolds are not manifolds. This is a consequence of the following characterization.

\begin{theorem}[Cannon \cite{Cannon}, Edwards \cite{Edwardsicm, EdwardsPreprint}] \label{thm:edwards}
Let $C$ be any homology-manifold. The following are equivalent:
\begin{compactenum}[\rm (i)] 
\item $C$ is a manifold; 
\item for every vertex $v$ of $C$, $\Lk(v,C)$ is simply connected.
\end{compactenum}
\end{theorem}

A famous instance of this criterion is the so called Double Suspension theorem. Recall that the \emph{suspension} of a simplicial complex $C$ is the complex 
\[ \Sigma(C) \; = \; \{x,\,y\} \ast C,\]
where $x$ and $y$ are two new vertices. In other words, $\Sigma(C)$ is the join of $C$ and a $0$-sphere.

\begin{cor}[Double Suspension theorem] \label{cor:edwards}
Let $H$ be any $d$-dimensional homology sphere. The double suspension $\Sigma^2(H) = \Sigma (\Sigma (H))$ is a manifold, homeomorphic to $S^{d+2}$. Since four edge links in $\Sigma^2(H)$ are combinatorially equivalent to $H$, if $H$ is not a PL sphere, then the sphere $\Sigma^2(H)$ is not a PL manifold.
\end{cor}

In a non-PL closed manifold, the subcomplex of faces whose link is not a PL sphere is called \emph{PL singular set}. Inside  $\Sigma^2(H) = S^0 \ast S^0 \ast H$, the PL singular set is $S^0 \ast S^0$, a $4$-cycle formed by the apices of the suspensions: In fact, the link of each edge in this cycle is $H$.

For each $d \ge 5$, the sphere $S^d$ admits both PL triangulations, like the boundary of the $(d+1)$-simplex, and non-PL ones, like the double suspension of some homology $(d-2)$-sphere (cf.~Corollary \ref{cor:edwards}).  One could wonder if there exist manifolds that admit only non-PL triangulations. This is indeed the case, as discovered by Kirby and Siebenmann:

\begin{example}[Kirby--Siebenmann \cite{KirSie}] \label{ex:KiSie}
For every $d\geq 5$, there exists a $d$-manifold that admits a triangulation, but does not admit any PL triangulation. In fact, any $d$-manifold $M$, with $d\geq 4$, can be PL triangulated if and only if the Kirby-Siebenmann invariant $\kappa(M)\in H^4(M,\mathbb{Z}/2\mathbb{Z})$ vanishes. This in particular implies that all manifolds with $H^4(M,\mathbb{Z}/2\mathbb{Z})=0$ can be PL triangulated. Thus, all contractible manifolds can be triangulated! In fact, Gleason even discovered that all contractible manifolds admit a smooth structure, cf.\ \cite{HundM}.
\end{example}

Clearly, by Theorem \ref{thm:pltr} the manifolds of Example \ref{ex:KiSie} do not admit any smooth structure.

\subsection{The case of 4-manifolds}  \label{sec:4m}

The $4$-dimensional case, as far as manifolds are concerned, is the most bizarre and complicated. On one hand, by the work of Perelman, all triangulations of all $4$-manifolds are PL (Theorem \ref{thm:Perelman}). Furthermore, every PL structure corresponds to a unique smooth structure (Theorem \ref{thm:UniqueSmoothStructure}). However, exotic structures on $4$-manifolds are far from being understood. One of the big remaining open problems in topology, is the question of whether exotic $4$-spheres exist.

\begin{problem}[Smooth Poincar\'e conjecture]  \label{conj:SPC}
Is there a manifold homeomorphic to $S^4$, but not diffeomorphic to it?
\end{problem}

Since smooth structures are in bijection with PL structures, one could rephrase the question above in the PL category as ``are there PL structures on $S^4$ other than the boundary of the $5$-simplex?''. By definition, PL spheres are precisely spheres that are PL-homeomorphic to the simplex boundary. In \cite{AB-SSZ} Adiprasito and the author characterized PL spheres as those that become shellable after sufficiently many barycentric subdivisions. Hence, a third way to rephrase Problem \ref{conj:SPC} is: ``Are there triangulations of $S^4$ that remain unshellable after any finite number of barycentric subdivisions?'' 

At the moment we have no invariant to distinguish smooth structures on $S^4$. (Very recently Rasmussen's splice
invariant has sparkled some optimism in this direction, cf. \cite{SPC}.)
Some evidence in favor of the smooth Poincar\'e conjecture is that all reasonable candidates for counterexamples have been proven standard: See Akbulut \cite{akb} and Gompf \cite{Gompf}. Some evidence against the conjecture, instead, is represented by the discovery that some $4$-manifolds have many exotic smooth structures. 
So, a priori there could be infinitely many exotic $4$-spheres; or some; or none. The only thing we know is that there can be at most countably many of them.

\subsection{Handle decompositions}
Handle decompositions can be viewed as a divide and conquer method: To understand an arbitrary $d$-manifold, we chop it into $d$-balls, which are glued together at $(d-1)$-submanifolds of their boundaries. These submanifolds have a prescribed topology, namely, they are required to be products of balls and spheres. 
 
By $\mathbb{B}^{d}$ we denote from now on  the unit ball in $\mathbb{R}^{d}$. 

\begin{definition}[Attaching a $0$-handle]
Let $M'$ be either the empty set, or a topological manifold (resp. a PL triangulation manifold, resp. a smooth manifold) of dimension $d$. 
Let $H^{(0)}$ be a topological manifold homeomorphic to $\mathbb{B}^d$ (resp. a PL triangulation of the $\mathbb{B}^d$, resp. a smooth manifold diffeomorphic to $\mathbb{B}^d$). We say that the disjoint union $M= M' \cup H^{(0)}$ is the result of \emph{attaching a $0$-handle to $M'$ in the topological} (resp. \emph{PL}, resp. \emph{smooth}) \emph{category}. 
\end{definition}

\begin{definition}[Attaching an $i$-handle] Let $1 \le i \le d$ be integers. Let $M'$ and $H^{(i)}$ be two $d$-dimensional topological manifolds (resp. PL manifolds, resp. smooth manifolds) with non-empty boundary. Suppose that:
\begin{compactitem}
\item $H^{(i)}$ is  homeomorpic (resp. PL homeomorphic, resp. diffeomorphic) to $\mathbb{B}^d$.
\item Some $(d-1)$-submanifold $A'$ of $\partial M'$ and some $(d-1)$-submanifold $A$ of $\partial H^{(i)}$ are both homeomorphic to $S^{i-1} \times \mathbb{B}^{d-i}$ (resp. combinatorially equivalent to each other and both homeomorphic to $S^{i-1} \times \mathbb{B}^{d-i}$, resp. diffeomorphic to $S^{i-1} \times \mathbb{B}^{d-i}$ via a diffeomorphism that can be extended to the interior of $M'$ and of $H^{(i)}$).
\end{compactitem}
Let $M$ be the $d$-manifold obtained by gluing $M'$ to $H^{(i)}$ via the identification $A' \equiv A$. We say that $M$ is the result of \emph{attaching an $i$-handle to $M'$ in the topological} (resp. \emph{PL}, resp. \emph{smooth}) \emph{category}. With slight abuse of notation, we will write $M = M' \cup H^{(i)}$ and say that $M$ is the union of $M$ and $H^{(i)}$.
\end{definition}

\begin{definition}[Handle decomposition]
Let $M$ be a topological $d$-manifold (resp. a PL $d$-manifold, resp. a smooth $d$-manifold) with possibly empty boundary. A \emph{topological} (resp.\ \emph{PL}, resp. \emph{smooth}) \emph{handle decomposition for $M$} is an expression of the form
\[
M \; = \;
H_0^{(0)} \, \cup \, H_1^{(i_1)} \cup \; \ldots \; \cup H_k^{(i_k)}
\]
where $H_0^{(0)}$ is a $0$-handle and every $H_h^{(i_h)}$ is an $i_h$-handle attached to the union of the previous ones, in the topological (resp. PL, resp. smooth) category. 
\end{definition}

For typographical reasons, we will often use the expression ``\emph{handles of index $i$}'' instead of ``$i$-handles''.

\begin{example}
For every $d$, the ball $\mathbb{B}^d$ admits a smooth handle decomposition into one handle of index $0$. (One can also come up with more complicated  decompositions: For example, one can obtain $\mathbb{B}^d$ also by attaching one $0$-handle, one $(d-1)$-handle, and one $d$-handle.) Conversely, any manifold with a smooth handle decomposition consisting of only one handle is, by definition, diffeomorphic to the unit ball in some Euclidean space.
\end{example}

\begin{example}
 The standard $d$-sphere $S^d$ admits a smooth handle decomposition into one handle of index $0$ and one handle of index $d$. It follows from the work by Cerf \cite{CerfPiso} on twisted diffeomorphisms that also all exotic $d$-spheres,  for $d\geq 7$, admit a smooth handle decomposition into one handle of index $0$ and one handle of index~$d$. 
\end{example}

\begin{example} 
If any exotic $4$-sphere exists, then it cannot have a smooth handle decomposition into one  $0$-handle and one $4$-handle. This is because all orientation-preserving diffeomorphisms $S^3\rightarrow S^3$ are isotopic to each other, by the work of Moise \cite{Moise}. (The smoothing of the attaching diffeomorphism is unique, by Theorem \ref{thm:UniqueSmoothStructure}.) It follows that if a closed $4$-manifold $M$ has a smooth handle decomposition into two handles, attached via some diffeomorphism $S^3\rightarrow S^3$, then $M$ is diffeomorphic to the standard $4$-sphere. 
This could be considered a reason why exotic $4$-spheres are difficult to find (if they exist): To construct them, one needs to deal with complicated smooth handle decompositions. 
\end{example}

Every manifold of dimension different from  $4$ admits a \emph{topological} handle decomposition. In contrast, the $4$-manifolds admitting topological handle decompositions are precisely those that admit a smooth structure.

\begin{example}
The $E_8$ $4$-manifold by Freedman does not admit any topological handle decomposition. 
\end{example}

\noindent In the next sections, we will see that all smooth manifolds admit a smooth handle decomposition (Theorem \ref{thm:smh1} and Proposition \ref{prp:smh2}), and that all PL triangulated manifolds admit a PL handle decomposition (Corollary \ref{PLtoPLhandle}).


\subsection{From Morse theory to smooth handles, and backwards} \label{sec:MorseHandle}
Let $M$ denote a closed smooth manifold. Let $f : M \rightarrow \R$ be a smooth function. 
As we know from calculus, the \emph{critical points} of $f$ are the points at which the gradient of $f$ vanishes. If $f$ is generic, it will have a finite number of isolated critical points (at which the Hessian of $f$ will be non-singular), and the critical points will have different images under $f$. 
Since $M$ is closed and compact, and $f$ is continuous, by Weierstrass' theorem $f$ must have a maximum and a minimum; so, if $\dim M > 0$, $f$ will have at least two critical points. From now on, a generic smooth function will be called a \emph{Morse function}, or sometimes a \emph{smooth Morse function}, to highlight the difference with discrete Morse theory.

\begin{definition}[Index of Critical Points]
Let $f : M \rightarrow \R$ be a Morse function. The index of a critical point $p \in M$ of $f$ is the dimension of the largest subspace of the tangent space $T_p M$ on which the Hessian of $f$ is negative definite. 
\end{definition}

Intuitively, the index is the number of independent directions around $p$ in which $f$ decreases. Any local minimum has index $0$; any local maximum has index equal to the dimension of $M$.  

\begin{lemma}[Morse] \label{lem:Morse}
Let $M$ denote a closed smooth manifold. Fix a generic smooth function $f : M \rightarrow \R$, and let $p_1, \ldots, p_k$ be its critical points. For any real number $a$, let us denote by $M_{\le a}$ the preimage under $f$ of the closed interval $(-\infty, a]$.  
\begin{compactenum}[\rm 1.]
\item Suppose the interval $[a,b]$ contains 
none of the critical values. Then $M_{\le b}$ is diffeomorphic to $M_{\le a}$.
\item Suppose the interval $[a,b]$ contains exactly one 
of the critical values, say, $f(p)$, in its interior. Let $i$ be the index of the critical point $p$. Then $M_{\le b}$ is diffeomorphic to the manifold obtained by attaching to $M_{\le a}$ a handle of index $i$ in the smooth category.
\end{compactenum}
\end{lemma}

\begin{proof}
See Milnor \cite{Milnor} or Matsumoto \cite{Matsumoto}.
\end{proof}
\enlargethispage{3mm}

\begin{theorem}\label{thm:smh1}
Let $M$ be a closed smooth manifold. Let $f : M \rightarrow \R$ be an arbitrary (smooth) Morse function. 
Then, $M$ admits a smooth handle decomposition into $c_i$ handles of index~$i$, where $c_i$ counts the number of index-$i$ critical points of $f$.
\end{theorem}

\begin{proof}
Let $p_1, \ldots, p_k$ be the critical points of $f$. Up to relabeling them, we can assume $f(p_1) < f(p_2) < \ldots < f(p_k)$. Choose $k+1$ real numbers $\lambda_0, \lambda_1, \ldots, \lambda_k$ such that
\[\lambda_0 \; < \; f(p_1) \; < \; \lambda_1 \; < \; f(p_2) \; < \; \lambda_2 \; < \ldots \; < \; \lambda_{k-1} \; < \; f(p_k) \; < \; \lambda_k .\]
Since the minimum and the maximum of $f$ on $M$ are points at which the gradient of $f$ vanishes, the minimum of $f$ on $M$ must be $p_1$, and the maximum of $f$ on $M$ must be $p_k$. In particular, $\lambda_0 \, < \, \{f(x): x \in M \}  \, < \, \lambda_k$.
Hence $M_{\le \lambda_0}$ is empty, while $M_{\le \lambda_k} = M$. The conclusion then follows by applying Lemma \ref{lem:Morse} to all the intervals $[\lambda_i, \lambda_{i+1}]$, in this order.
\end{proof}

The converse also holds: Every smooth handle decomposition comes from a Morse function, as the following, well-known result shows.

\begin{prp}\label{prp:smh2}
Let $M$ be a closed smooth manifold. If $M$ has a smooth handle decomposition into $c_i$ handles of index~$i$, some Morse function on $M$ has $c_i$ critical points of index~$i$.
\end{prp}


One can extend to manifolds with boundary the definition of smooth Morse vector, as follows:

\begin{definition}
A manifold with boundary has \emph{Morse vector $(c_0,c_1,\ldots, c_d)$ } if it admits  a smooth handle decomposition with $c_i$ handles of index~$i$. 
\end{definition}

By Theorem \ref{thm:smh1} and Proposition \ref{prp:smh2}, if $M$ has empty boundary this definition coincides with the one in terms of the number of critical $i$-cells of a Morse function.

\subsection{From smooth to PL handles, and backwards}

Let us quickly review the relation between Morse theory and PL handles theory.

\begin{rem}[Smooth handle decompositions are also PL handle decompositions] \label{rem:StoPLhandle}
Suppose a smooth manifold $M$ (with or without boundary) has a smooth handle decomposition into smooth handles $H^{\operatorname{s}}_i$. 
By Theorem \ref{thm:pltr}, 
both handles can be PL triangulated to obtain a family of PL handles $H^{\operatorname{PL}}_i$ that are
attached along PL homeomorphisms of their boundary. The manifold resulting from gluing the $H^{\operatorname{PL}}_i$ is $M$, and the resulting triangulation is PL homeomorphic to the standard triangulation of $M$ obtained from Theorem \ref{thm:pltr}.
\end{rem}

The converse of Remark \ref{rem:StoPLhandle} is more delicate: As we have seen in Section \ref{sec:PLtoSmooth}, associated to a given PL manifold there might be no smooth structure, or more than one. However, in low dimensions, the obstruction theories of Munkres \cite{Munkressmooth} and Hirsch \cite{HundM} yield the following:

\begin{prp}[Hirsch \cite{HundM}, Munkres \cite{Munkressmooth}]
Let $M$ be a PL manifold of dimension $d\leq 7$, with or without boundary, that admits a PL handle decomposition into $c_i$ handles of index~$i$. Then $M$ admits a smooth handle decomposition with $c_i$ handles of index~$i$.
\end{prp}

Combining this with Remark \ref{rem:StoPLhandle}, Theorem \ref{thm:smh1} and Proposition \ref{prp:smh2}, we obtain:

\begin{theorem} \label{thm:smoothtoPL}
Let $M$ denote a manifold of dimension $d\leq 7$ (with or without boundary). The following are equivalent:
\begin{compactenum}[\rm (i)]
\item $M$ admits a smooth Morse function with $c_k$ critical points of index~$k$. 
\item $M$ admits a smooth handle decomposition with $c_k$ handles of index~$k$.
\item $M$ admits a PL handle decomposition with $c_k$ handles of index~$k$.
\end{compactenum}
For $d> 7$, one has {\rm (i)} $\Leftrightarrow$ {\rm (ii)} $\Rightarrow$ {\rm (iii)}, but the converse of the last implication is false. 
\end{theorem}

\section{Main results}\label{sec:PLdiscrete}
Recall that a smooth $d$-manifold has \emph{Morse vector $(c_0,c_1,\ldots, c_d)$ } if it admits a smooth handle decomposition with exactly $c_i$ handles of index~$i$. If the manifold is closed, this is equivalent to admitting a Morse function with exactly $c_i$ critical points of index $i$. Analogously, a $d$-dimensional simplicial complex has \emph{discrete Morse vector $(c_0,c_1,\ldots, c_d)$ } if it admits a discrete Morse function with exactly $c_i$ critical $i$-cells. Finally, we say that a PL manifold has  \emph{PL handle vector $(c_0,c_1,\ldots, c_d)$ } if it admits a PL handle decomposition with exactly $c_i$ handles of index~$i$. 

As we explained in the previous section, it is known that 
\[ 
\left \{\textrm{Morse vectors of M} \right \}
\subset 
\left \{\textrm{PL handle vectors of M} \right \},
\]
with equality if the dimension of $M$ is at most seven. In this section, we prove the following result:

\begin{theorem} \label{thm:EquivalenceOfMorseVectors}
For any PL $d$-manifold $M$, with or without boundary, 
{\em
\[ 
\big \{\textrm{PL handle vectors of M} \big \}
=
\Big \{\!\!\begin{array}{c}
	\textrm{discrete Morse vectors} \\
	\textrm{on PL triangulations of M} 
	\end{array} \!\!\Big\}
\subset
\Big \{\!\!\begin{array}{c}
	\textrm{discrete Morse vectors} \\
	\textrm{on triangulations of M} 
	\end{array} \!\!\Big\},
\]
}
and the inclusion on the right is not an equality in general.
\end{theorem}

The proof will be organized as follows. 
In Section \ref{subs:back} we will show  that every discrete Morse vector on a PL triangulation of $M$ is also a PL handle vector. In Section \ref{subs:1} we will prove the converse, namely, that every PL handle vector is also a discrete Morse vector on \emph{some} PL triangulation. (In fact, we show something stronger, namely, that \emph{any} PL triangulation can be used, up to subdividing it barycentrically a suitable number of times.) Finally, Example \ref{ex:AdiprasitoBenedetti} is a $5$-manifold on which $(1,0,0,0,0,0)$ is a discrete Morse vector (on some non-PL triangulation), but not a PL handle vector.

\subsection{From PL discrete Morse vectors to PL handles} \label{subs:back}
Recall that a discrete Morse function on a complex $C$ is any map $f : C \rightarrow \mathbb{Q}$ that satisfies: 
\begin{compactenum}[(i)]
\item (Monotonicity) If $\sigma \subset \tau$, then $f(\sigma) \le f(\tau)$.
\item (Semi-injectivity) For each $q \in \mathbb{Q}$, the cardinality of $f^{-1} (q)$ is at most $2$.
\item (Genericity) If  $f(\sigma) = f(\tau)$, then either $\sigma \subset \tau$ or $\tau \subset \sigma$.
\end{compactenum}
\noindent The \emph{critical faces} of $C$ for $f$ are the faces at which $f$ is injective. 
Let us denote by $\sd^m C$ the $m$-th iterated barycentric subdivision of a complex $C$,  recursively defined as $\sd^m C \, = \, \sd ( \sd^{m-1} C )$.
Our first goal is to prove the following result:

\begin{theorem}\label{thm:dmttopl}
If a topological manifold has a PL triangulation $M$ on which some discrete Morse function has $c_i$ critical $i$-faces, then for some $r$ the complex $\sd^r M$ has a PL handle decomposition with $c_i$ $i$-handles.
\end{theorem}

The proof of Theorem \ref{thm:dmttopl}, which reflects joint work with Adiprasito, relies on three results in PL topology. Two of them are very old and due to Whitehead (Lemma \ref{lem:RegNeigh} and Theorem \ref{thm:WTHM}); the other is very recent (Theorem \ref{thm:subdivision}). Aside from these tools, the proof is inspired by the 2009 paper by Jer\u{s}e and Mramor-Kosta \cite{JerseMramorKosta}, which shows how to divide a (PL or non-PL) triangulation into ``descending regions''. These regions may be lower-dimensional, but there is a standard trick to fix this, namely, one can thicken them by passing to regular neighborhoods, as we will now explain. If $D$ is a subcomplex of a manifold $M$, the \emph{$m$-th derived neighborhood of $D$ in $M$}, usually denoted by $N^m(D,M)$, is the subcomplex of $\sd^m (M)$ formed by the cells whose closure intersects $|D|$. 
For any $m \ge 2$, the underlying spaces of $N^m(D,M)$, is called in the literature \emph{a regular neighborhood} of $D$ in $M$.
The regular neighborhood of any subcomplex collapses onto it after appropriate subdivision:

\begin{lemma}[{Whitehead, cf. Glaser \cite[Lemmas III.9 and III.10]{Glaser}}]\label{lem:RegNeigh}
Let $C$ be a subcomplex of a (PL or non-PL) triangulated manifold $M$. For each $m \ge 2$, the $m$-th derived neighborhood of $C$ collapses onto $\sd^m (C)$. 
\end{lemma}

In particular, if we have a discrete Morse function defined on a subcomplex of some PL manifold, we can lift it to a discrete Morse function defined on a regular neighborhood of the subcomplex:

\begin{lemma}\label{lem:pusd}
Let $C$ be a subcomplex of a PL triangulated manifold $M$. Suppose $C$ admits a discrete Morse function with $c_i$ critical $i$-faces. For any $n\ge 0$ and $m \ge 2$ integers, also the complex $\sd^n N^m(C,M)$ admits a discrete Morse function with $c_i$ critical $i$-faces. 
\end{lemma}

\begin{proof} The barycentric subdivision of $C$ also admits a discrete Morse function with $c_i$ critical $i$-faces; and by induction, so does $\sd^{n+m} C$. By Lemma \ref{lem:RegNeigh}, the regular neighborhood $\sd^n N^m(C,M)=N^m(\sd^n C, \sd^n M)$ collapses onto $\sd^{n+m} C$, whence we conclude.
\end{proof}

Collapses are not internal operations with respect to the class of triangulated manifolds. However, using regular neighborhoods, as explained for example in Rourke--Sanderson \cite{RourkeSanderson}, it is possible to introduce a structure that allows us to speak in the language of manifolds.

\begin{definition}[{cf.~Rourke--Sanderson \cite[p.~40]{RourkeSanderson}}]
Let $M$ be a PL triangulated $d$-manifold. Let $\Delta$ be a $d$-face of $M$. Let $M'$ be the submanifold of $M$ that contains all $d$-faces of $M$, except $\Delta$. If $M$ collapses onto $M'$, the reduction from $M$ to $M'$ is called \emph{elementary shelling}.  We say that $M$ \emph{shells to} $M'$ if some sequence of elementary shelling reduces $M$ to $M'$.
\end{definition}

\begin{rem} \label{rem:NonHomeomorphic}
Each elementary shelling maintains the homeomorphism class. So, if $M$ shells to $M'$, then $M$ and $M'$ are homeomorphic. In contrast, if a manifold $M$ collapses onto a manifold $M'$, then $M$ and $M'$ need not be homeomorphic: This can be seen by barycentrically subdividing the $5$-manifold in Example \ref{ex:AdiprasitoBenedetti} below. The result is a non-PL $5$-manifold that collapses onto a $5$-ball, without being homeomorphic to it.
\end{rem}

In the following, a \emph{polytopal complex} is a finite non-empty collection $C$ of polytopes in some $\mathbb{R}^k$, such that any two of them intersect at a common face, and every face of any polytope in $C$ is also an element of $C$.

\begin{theorem}[{Whitehead, cf.~Rourke--Sanderson \cite[Theorem 3.26 \& Cor. 3.27]{RourkeSanderson}}]\label{thm:WTHM}
Let $C$ be a subcomplex of a PL triangulated manifold $M$. Suppose that $C$ collapses to some subcomplex $D$. Then a regular neighborhood of $C$ shells to a regular neighborhood of $D$. In case $D$ is a single vertex, this regular neighborhood of $C$ is a PL ball.
\end{theorem}

\begin{theorem}[{Adiprasito--Benedetti \cite[Theorem 5.3]{AB-SSZ}}] \label{thm:subdivision}
Let $C$, $D$ be polytopal complexes. Suppose that some subdivision $C'$ of $C$ shells to some subdivision $D'$ of $D$. Then, for $n$ large enough, $\sd^n C$ shells to $\sd^n D$.
\end{theorem}

By merging the latter two theorems, we obtain the following consequence:
\enlargethispage{3mm}

\begin{cor}\label{cor:shellbary}
Let $C$ be a subcomplex of a PL manifold $M$. Suppose $C$ collapses to some subcomplex $D$. Then for $n$ and  $m$ large enough, $\sd^n N^m(C,M)$ shells to $\sd^n N^m(D,M)$.
\end{cor}

Next, we recall a well-known fact in PL topology. Given a simplicial complex $C$, a subcomplex $D$ of $C$ of is called {\em full} if any face $\sigma$ of $C$ whose vertices are all in $D$ belongs to $D$ as well. For example, for any complex $C$ and for any subcomplex $D$ of $C$, the barycentric subdivision $\sd \, D$ is always full in $\sd\, C$ \cite[Sec.\ 3]{ZeemanBK}.

\begin{lemma}\label{prp:NewC} Let $k < d$ be non-negative integers. 
Let $C$ be a $k$-dimensional full subcomplex of a PL $d$-manifold $M$. Let $D$ be the subcomplex of $C$ obtained by removing from $C$ a  $k$-cell $\sigma$. Then $N^2(C,M)$ is obtained from $N^2(D,M)$ by attaching a $k$-handle.
\end{lemma}

We are now ready to prove the main theorem of this section.

\begin{proof}[{\bf Proof of Theorem \ref{thm:dmttopl}}]
We proceed by induction on the number $|\mathbf{c}|$ of critical cells of the discrete Morse function $f$ on $M$. 

\begin{compactdesc}
\item[$|\mathbf{c}|=1$] If $M$ has only one critical face, then this face is a single vertex. By Whitehead's Theorem \ref{thm:WTHM}, $M$ is a PL ball. 
\item[$|\mathbf{c}|>1$] Let us assume the claim holds for PL manifolds that admit a discrete Morse function with less than $k$ critical faces. Let $M$ be a PL triangulated manifold with exactly $k$ critical faces. Let $\tau$ be the critical face with the highest value under $f$, and let $t = \dim \tau$ be its dimension.
Let $\varTheta$ denote the subcomplex of $M$ of faces with a value lower than $f(\tau)$ under $f$. By Corollary \ref{cor:shellbary} and Lemma \ref{prp:NewC}, for large $n$ and $m$, the complex $\sd^n N^m(M,M)$ can be decomposed into the complex $\sd^n N^m(\varTheta,M)$ with a PL $t$-handle attached. By Lemma \ref{lem:pusd}, $\sd^n N^m(\varTheta,M)$ admits a discrete Morse function with $c_i$ critical cells of dimension $i\neq t$, and $c_t-1$ critical cells of dimension $t$. By the inductive assumption, $\sd^n N^m(\varTheta,M)$ admits the desired PL handle decomposition, and so does $M$.
\end{compactdesc}
\end{proof}

\begin{cor} \label{PLtoPLhandle}
Every PL manifold admits a PL handle decomposition (possibly after some iterated barycentric subdivisions). 
\end{cor}

\begin{proof}
Any simplicial complex admits a discrete Morse function, possibly with many critical faces. If the complex is a PL triangulation, we can conclude via Theorem \ref{thm:dmttopl}.
\end{proof}

\subsection{Intermezzo: Duality for Morse functions}
The Betti numbers of any \emph{closed} manifold\footnote{In the present paper, for simplicity, all manifolds are assumed to be orientable. For non-orientable manifolds, Poincar\'e duality holds only if one takes homology and cohomology with coefficients in the integers mod $2$. For further details, see Hatcher \cite[Sections 3.3 \& 3.H]{Hatcher}.} are palindromic, i.e. 
\[ (\beta_0, \beta_1, \ldots, \beta_{d-1}, \beta_{d}) 
\: = \:(\beta_d, \beta_{d-1}, \ldots, \beta_{1}, \beta_{0}). \]
This is a consequence of Poincar\'e duality, a well-known property of closed manifolds. In the language of Morse theory, the property can be stated as follows: Suppose a smooth, closed manifold $M$ admits a generic smooth function $f$, with $c_i$ critical points of index $i$. Then $-f$ is also a generic smooth Morse function on $M$, and it has $c_{d-i}$ critical points of index $i$. In fact, each index-$i$ critical point of $f$ is also a critical point for $-f$, with index $d-i$.

\begin{cor} \label{cor:SmoothDuality} Let $M$ be a closed topological manifold.
\begin{compactenum}[ \rm (1) ]
\item If $(c_0, c_1, \ldots, c_{d-1}, c_{d})$ is a Morse vector on $M$, so is 
$(c_d, c_{d-1}, \ldots, c_{1}, c_{0})$ (with respect to the same smooth structure). 
\item If $(c_0, c_1, \ldots, c_{d-1}, c_{d})$ is a PL handle vector on $M$, so is 
$(c_d, c_{d-1}, \ldots, c_{1}, c_{0})$ (with respect to the same PL structure).
\end{compactenum}
\end{cor}

There is also a version of Poincar\'e duality for triangulated manifolds, due to Forman. It involves the notion of \emph{dual block complex} $T^*$ of a triangulation $T$, for which we refer the reader to Munkres' book \cite[p.~377]{Munkres}. The ``blocks'' forming the complex are PL-homeomorphic to the stars of the faces of $T$, so they might not be homeomorphic to balls if $T$ is not PL. However, if $T$ is PL, then the dual blocks are balls, and $T^*$ is a regular CW-complex. In particular, the barycentric subdivision of $T^*$ is a simplicial complex, combinatorially equivalent to the barycentric subdivision of $T$.

\begin{prp}[Forman] \label{prp:dualityForman}
If a triangulated, closed manifold $M$ (not necessarily PL) admits a discrete Morse function $f$ with $c_i$ critical $i$-faces, then the dual block complex $M^*$ admits some discrete Morse function with $c_{d-i}$ critical $i$-faces. (For example, take on $M^*$ the function 
\[\sigma^* \; \longmapsto \; - f (\sigma),\]
that assigns the value $- f (\sigma)$ to the dual block of each face $\sigma$. With slight abuse of notation, this function is usually denoted by $-f$.) 
\end{prp}

\begin{cor} \label{cor:dualityForman} Let $M$ be a closed PL manifold. If $(c_0, c_1, \ldots, c_{d-1}, c_{d})$ is a discrete Morse vector on some PL triangulation of $M$, then so is 
$(c_d, c_{d-1}, \ldots, c_{1}, c_{0})$.
\end{cor}

\begin{proof}
By Proposition \ref{prp:dualityForman}, $(c_d, c_{d-1}, \ldots, c_{1}, c_{0})$ is a discrete Morse vector on $M^*$, and thus also on $\sd M^* = \sd M$, which is also PL.
\end{proof}

Corollary \ref{cor:SmoothDuality}, Proposition \ref{prp:dualityForman} and Corollary \ref{cor:dualityForman} do not extend to manifolds with boundary. For example, $(1, 0, \ldots, 0)$ is the handle vector of the $d$-ball, but read from right to left this is not a valid handle vector, because it starts with a $0$. To extend the idea of duality to manifolds with boundary, one usually uses the notion of (smooth or PL) handle decomposition of a \emph{cobordism}: For details, see Rourke--Sanderson \cite{RourkeSanderson} or Sharko \cite{Sharko}.

In discrete Morse theory, there is a rather direct way to extend duality to manifolds with boundary. This was introduced by the author in \cite{B-DMT4MWB}, to which we refer for the definition of dual block complex $M^*$ in case $\partial M \ne \emptyset$. (Caveat: If the boundary is non-empty, $(M^*)^*$ does not coincide with $M$, but it is basically a collar of $M$.)

\begin{definition}[Boundary--critical, interior Morse vector] A discrete Morse function on a manifold with (possibly empty) boundary is called \emph{boundary-critical} if all the boundary faces are critical. A boundary-critical discrete Morse function $f$ has \emph{interior Morse vector} $(c_0, c_1, \ldots, c_{d-1}, c_{d})$, if $f$ has exactly $c_i$ critical faces in the interior of the manifold.
\end{definition}

\begin{prp} [Benedetti \cite{B-DMT4MWB}]\label{prp:dualityMe}
If a PL triangulated manifold $M$ admits a discrete Morse function $f$ with $c_i$ critical $i$-faces and no critical boundary faces, then the dual block complex $M^*$ admits a boundary-critical discrete Morse function with $c_{d-i}$ critical interior $i$-faces. 
Conversely, if $M$ admits a boundary-critical discrete Morse function $f$, with $c_i$ critical interior $i$-faces, then $M^*$ admits a discrete Morse function with $c_{d-i}$ critical $i$-faces. 
\end{prp}

\begin{definition}[Endocollapsible]
Let $M$ be a triangulated $d$-manifold with non-empty boundary. $M$ is called \emph{endocollapsible} if $M$ minus a facet collapses onto the boundary of $M$.
\end{definition}

It is easy to see that all shellable balls are endocollapsible; the converse is false, cf. \cite{B-DMT4MWB}. 

\begin{cor} \label{cor:dualityMe} If $(c_0, c_1, \ldots, c_{d-1}, c_{d})$ is a discrete Morse vector (resp. an interior discrete Morse vector) on some PL triangulated manifold $M$, then 
$(c_d, c_{d-1}, \ldots, c_{1}, c_{0})$ is an interior discrete Morse vector (resp. a discrete Morse vector) on some other PL triangulation. One can prescribe the latter triangulation to be an iterated barycentric subdivision of $M$.
\end{cor}

\begin{proof} 
Here we no longer have that $\sd M^* = \sd M$, so we cannot conclude as in Corollary \ref{cor:dualityForman}. However, since $M$ is PL, all face stars in $M$ are PL balls. So, for suitably large $r$, all face stars in $M'=\sd^r M$ are shellable balls \cite{AB-SSZ}. In particular, all face stars in $M'$ are collapsible and endo-collapsible. The conclusion follows then on $\sd M' = \sd^{r+1} M$ by \cite[Theorem 3.20]{B-DMT4MWB}.
\end{proof}

Finally, we mention that all shellable manifolds with non-empty boundary are homeomorphic to balls. This is still true with the weaker endo-collapsibility assumption:

\begin{theorem}[Ball theorem \cite{B-DMT4MWB}] \label{thm:BallTheorem}
If a manifold with non-empty boundary admits an endocollapsible triangulation, then it is homeomorphic to a ball.
\end{theorem}

\begin{cor}[Whitehead] \label{cor:Whitehead} 
If a manifold admits a collapsible PL triangulation, then it is homeomorphic to a ball.
\end{cor}

\begin{proof}
The claim follows at once from Whitehead's regular neighborhood theory (Theorem \ref{thm:WTHM}). Here is an alternative proof: 
Some subdivision of $M$ is endo-collapsible by Corollary \ref{cor:dualityMe}. Via Theorem \ref{thm:BallTheorem} we conclude. See also Remark \ref{rem:WhiteheadStrengthened}.
\end{proof}

\begin{rem}
It is easy to extend the definition of ``endocollapsible'' from simplicial complexes to polytopal complexes, and even to complexes formed by homotopy cells or homology cells (like the aforementioned ``dual block complex'' of a triangulation). This way, one could extend Corollary \ref{cor:dualityMe} to non-PL triangulations of manifolds, and prove for example that if $M$ is collapsible (but not PL), then $M^*$ is an endocollapsible homotopy-cell complex (but not a cell complex).
However, we stress that in this broader generality Theorem \ref{thm:BallTheorem} is no longer true. We will see in Remark \ref{rem:WhiteheadStrengthened} that some collapsible triangulated manifold $M$ is not homeomorphic to a ball; so the dual block complex $M^*$ is ``endo-collapsible'' in the above sense, but not homeomorphic to a ball. 

In contrast, Theorem \ref{thm:BallTheorem} does extend to the generality of polytopal or regular cell complexes. The proof is elementary: A single barycentric subdivision preserves both homeomorphism type and endo-collapsibility, yet transforms every regular cell complex into a simplicial complex (to which we can apply Theorem \ref{thm:BallTheorem}.) 
\end{rem}

\subsection{From PL handles to discrete Morse vectors} \label{subs:1}
Here we show that every PL handle vector is a discrete Morse vector on \emph{some} iterated barycentric subdivision (Theorem \ref{thm:PLtoDiscrete}). This reviews and improve the work by Gallais \cite{Gallais} (cf. Remark \ref{rem:gallaisgap} below).
Our proof relies on the following Lemma:

\begin{lemma}[{Adiprasito--Benedetti \cite{AB-SSZ}}] \label{lem:ShellableSubdivision}
For any PL triangulation $B$ of a ball, there is an integer $r$ such that $\sd^r B$ is shellable.
\end{lemma}

Given a PL handle decomposition of an arbitrary PL manifold $M$, up to subdividing $M$ barycentrically suitably many times, by Lemma \ref{lem:ShellableSubdivision} we can assume that all handles are shellable. We will now explain how to derive from this a global discrete Morse function. For this we need a few lemmas.

\begin{definition}[Perfect Morse function] A discrete Morse function $f$ on a complex $C$ is called \emph{perfect} if the Morse vector of $f$ coincides with the Betti vector of $C$ (that is, if $f$ has exactly $\beta_i (C)$ critical $i$-faces). Similarly, a boundary--critical discrete Morse function $f$ on a manifold with boundary $M$ is called \emph{perfect} if the interior Morse vector of $f$ is the mirror image of the Betti vector of $M$ (that is, if $f$ has exactly $\beta_i (M)$ critical interior $(d-i)$-faces, where $d=\dim M$).
\end{definition}

For example, if $B$ is a triangulation of a ball (not necessarily PL), $B$ admits a perfect discrete Morse function
if and only if $B$ is collapsible; and $B$ admits a perfect boundary-critical discrete Morse function if and only if $B$ is endo-collapsible.

\begin{definition} Let $C_1$ and $C_2$ be disjoint complexes.  Let $g_i$ ($i=1,2$) be a discrete Morse function on $C_i$. Up to replacing $g_2$ with itself minus a large constant, we can assume that the maximum value of $g_2$ on $C_2$ is smaller than the minimum value of $g_1$ on $C_1$. The \emph{disjoint union} $f = g_1 \sqcup g_2$ is the function $f: C_1 \sqcup C_2 \rightarrow \mathbb{R}$ defined by
\[ f (\sigma) \ = \
\left \{
\begin{array}{cc}
g_1 (\sigma) & \textrm{ if } \sigma \textrm{ is in } C_1, \\ 
g_2 (\sigma) & \textrm{ if } \sigma \textrm{ is in } C_2.
\end{array} 
\right. \]
\end{definition}

\begin{lemma} \label{lem:DisjointUnion}
The function $g_1 \sqcup g_2$ is a discrete Morse function. Its critical $i$-faces are precisely the critical $i$-faces of $g_1$ plus the critical $i$-faces of $g_2$. In particular, if the $C_i$'s are (pseudo)manifolds and both $g_i$'s are boundary-critical, so is $g_1 \sqcup g_2$. 
\end{lemma}

From now on, by $\II$ we denote the unit interval $[0,1]$ in $\mathbb{R}$.

\begin{lemma} \label{lem:HandleAttachment0}
Let $M$ be a PL triangulation of a $d$-manifold with boundary. \\Suppose $M=B_1 \cup B_2$, where 
\begin{compactenum}[\rm (i)]
\item $B_1$ and $B_2$ are shellable (or endocollapsible) $d$-balls, 
\item $B_1 \cap B_2$ is a $(d-1)$-dimensional subcomplex of $\partial B_1$ and of $\partial B_2$, and
\item $B_1 \cap B_2$ is homeomorphic to $S^{j-1} \times \II^{d-j}$, for some $j \in \{1, \ldots, d\}$. 
\end{compactenum}
Then any boundary-critical discrete Morse function $h$ on $B_1 \cap B_2$ lifts to a boundary-critical discrete Morse function $u$ on $B_1 \cup B_2$, with the critical interior cells of $u$ being exactly those of $h$, minus a $(d-1)$-face, plus a $d$-face. 
\end{lemma}

\begin{proof}
Let $\sigma$ be any $(d-1)$-face of $B_1 \cap B_2$. Let  $\Sigma_1, \Sigma_2$ be the two $d$-faces of $M$ containing $\sigma$; clearly one of them belongs to $B_1$ and the other to $B_2$. Up to relabeling them, we can assume $\Sigma_i \in B_i$. Since $B_1$ is endo-collapsible, $B_1 - \Sigma_1$ collapses onto $\partial B_1$. Since all free faces removed in such a collapse are internal, the gluing of $B_2$ onto a portion of $\partial B_1$ does not destroy the freeness of any face. Therefore,
\begin{equation}\label{eq:1}
B_1 \cup B_2  - \Sigma_1 \: \textrm{ collapses onto } \: \partial B_1 \, \cup \, B_2.
\end{equation}
Inside $\partial B_1 \, \cup \, B_2$ there is now only one $d$-face left that contains $\sigma$, namely, $\Sigma_2$. Hence
\begin{equation}\label{eq:2}
\partial B_1 \, \cup \, B_2 \: \textrm{ collapses onto } \: \partial B_1 \, \cup B_2 - \Sigma_2 - \sigma.
\end{equation}
Since $B_2$ is endocollapsible, $B_2 - \Sigma_2$ collapses onto $\partial B_2$. Equivalently, $B_2 - \Sigma_2 - \sigma$ collapses onto $\partial B_2 - \sigma$. This implies that
\begin{equation}\label{eq:3}
\partial B_1 \, \cup B_2 - \Sigma_2 - \sigma \: \textrm{ collapses onto } \:\partial B_1 \cup \partial B_2 - \sigma.
\end{equation}
Let $C := \partial B_1 \cup \partial B_2$. 
Using $h$, we can reduce $C - \sigma$ further. In fact, $C$ can be viewed as the result of attaching $(B_1 \cap B_2)$ onto $\partial  (B_1 \cup B_2)$ along its boundary $\partial (B_1 \cap B_2)$. Since $h$ is boundary-critical, it reduces $B_1 \cap B_2 - \sigma$ via collapses and deletions onto $\partial (B_1 \cap B_2)$. The attachment  of $\partial (B_1 \cup B_2)$ does not affect the freeness of any face in such reduction; therefore, the same sequence of collapses and deletions reduces $C - \sigma$
onto 
$\partial (B_1 \cap B_2) \, \cup \, \partial (B_1 \cup B_2)$,
which is simply $\partial (B_1 \cup B_2)$, or in other words, $\partial M$. 

In conclusion, by Equations \ref{eq:1}, \ref{eq:2}, and \ref{eq:3} the complex $M - \Sigma_1$ collapses onto $C - \sigma$; using the function $h$ we have reduced $C -\sigma$ further to $\partial M$.
This yields a boundary-critical discrete Morse function $u$ on $M$. The critical interior cells of $h$ and $u$ are by construction the same, with two exceptions: $\Sigma_1$ is a critical face of $u$, but not of $h$, because it does not belong to the intersection $B_1 \cap B_2$; $\sigma$ is a critical face of $h$, but not of $u$, because $u(\sigma) = u(\Sigma_2)$.
\end{proof}

\begin{lemma} \label{lem:HandleAttachment}
Let $M$ be a PL triangulation of a $d$-manifold with boundary. \\Suppose $M=M' \cup B$, where 
\begin{compactenum}[\rm (i)]
\item $M'$ is a PL triangulation of a $d$-manifold with boundary and $B$ is an endocollapsible $d$-ball;
\item $M'$ admits a perfect boundary-critical discrete Morse function $g$;
\item $M' \cap B$ is a $(d-1)$-dimensional subcomplex of $\partial M'$ and of $\partial B$;
\item $M' \cap B$ is homeomorphic to $S^{j-1} \times \II^{d-j}$, for some $j \in \{1, \ldots, d\}$. 
\end{compactenum}
Then any boundary-critical discrete Morse function $h$ on $M' \cap B$ lifts to a boundary-critical discrete Morse function $u$ on $M' \cup B$, with the critical interior cells of $u$ being exactly those of $h$, minus a $(d-1)$-face, plus all the critical interior cells of $g$. 
\end{lemma}

\begin{proof}
The case where $M'$ is a ball has already been treated in Lemma \ref{lem:HandleAttachment0}. 
If $M'$ is not a ball, we proceed analogously to the proof of Lemma \ref{lem:HandleAttachment0}, but instead of starting with a collapse of $M' - \Delta$ onto $\partial M'$, we use the boundary-critical function $g$ on $M'$. For the rest, the proof is the same; in the last step, when we determine the number of critical cells of $u$, we have to include in the count the critical cells of $g$. 
\end{proof}


\begin{theorem} \label{thm:PLtoDiscrete}
Let $d$ be a positive integer. Let $M$ be any PL triangulation of any $d$-manifold with boundary. Assume $M$ has a PL handle decomposition into $c_i$ PL handles of index~$i$. For $r$ large enough, one can define on $\sd^r M$ 
\begin{compactenum}[  \rm (1) ]
\item a discrete Morse function with $c_i$ critical $i$-faces, and
\item a boundary-critical discrete Morse function with $c_{d-i}$ critical interior $i$-faces.
\end{compactenum}
\end{theorem}

\begin{proof} 
By Corollary \ref{prp:dualityMe}, it suffices to prove the second claim. 
We proceed by double induction on the total number of handles and on the dimension of $M$. If $M$ is $1$-dimensional, then $M$ is either a path or a cycle, and  the claim is obvious. If $M$ has a handle decomposition into one handle only, then $M$ is a PL $d$-ball, and by Lemma \ref{lem:ShellableSubdivision} there is an $r$ such that $\sd^r M$ is shellable. In particular, $\sd^r M$ is 
endocollapsible (or equivalently, it admits a boundary-critical discrete Morse function with one critical interior $d$-face). 

So, assume that $d \ge 2$ and that $M$ has a handle decomposition into $\sum c_i$ handles, with  $\sum c_i \ge 2$. Let $B$ be the last handle in a handle decomposition of $M$, and let $j$ be the index of~$B$. Since $M$ is connected, we can assume $c_0 = 1$ and $j \ge 1$. 
Let $\textbf{c'}$ be the vector defined by $c'_i + \delta_{i,j} = c_i$ for all $i$, where $ \delta_{i,j}$ is the Kronecker delta. (In other words, $\textbf{c'}= \textbf{c} - \textbf{e}_j$.) Then $M$ decomposes as  
\[ M = M' \cup B,\] where $M'$ is a $d$-manifold with a handle decomposition into  $c'_i$ handles of index~$i$. We are going to apply the inductive assumption to the two manifolds $M'$ and $M' \cap B$: In fact, $M' \cap B$ has smaller dimension than $M$, and $M'$ decomposes into fewer handles than $M$.

  If $j=1$, then $M' \cap B$ is a disjoint union of two $(d-1)$-balls, and there is $s$ such that the $s$-th barycentric subdivision of both balls is endo-collapsible. Via Lemma \ref{lem:DisjointUnion}, this yields a boundary-critical discrete Morse function on $\sd^s (M' \cap B)$ with two critical interior $(d-1)$-faces. If instead $j > 1$, then $M' \cap B$ is a PL triangulation of $S^{j-1} \times \II^{d-j}$. As topological space,  $S^{j-1} \times \II^{d-j}$ admits a topological handle decomposition into one $0$-handle and one $(j-1)$-handle. Up to subdividing it barycentrically $k$ times, we can assume that $M' \cap B$ has a PL handle decomposition with $a_i$ handles of index~$i$, where $a_0 = 1$, $a_1 = \ldots = a_{j-2} = 0$, and $a_{j-1} = 1$. Since $M' \cap B$ is $(d-1)$-dimensional, by the inductive assumption there is an integer $s$ such that $\sd^{s} (M' \cap B)$ has a boundary-critical discrete Morse function $f$ with $a_{(d-1)-i}$ critical interior $i$-faces. Putting the two cases together, for any $j \ge 1$ the triangulation $\sd^{s} (M' \cap B)$ admits a boundary-critical discrete Morse function with one critical interior $(d-1)$-face, one critical interior $(d-j)$-face, and no other critical interior face. As for $M'$, by the inductive assumption we can find an integer $r'$ and a boundary-critical discrete Morse function $g$ on $\sd^{r'} M'$ with exactly $c'_{d-i}$ critical interior $i$-faces.  Moreover, by Lemma \ref{lem:ShellableSubdivision} there is an integer $t$ such that $\sd^t B$ is shellable.

We are now in the position of applying Lemma \ref{lem:HandleAttachment}. Set $r := \max(r',s+k,t)$. Let us subdivide $M$ barycentrically $r$ times. Then 
\[\sd^r M = \sd^r M' \cup \sd^r B,\] where $\sd^r B$ is shellable, $\sd^r M'$ admits a perfect boundary-critical discrete Morse function $g$, and the intersection $\sd^r M' \cap \sd^r B = \sd^r (M' \cap B)$ admits a perfect boundary-critical discrete Morse function $h$. 
By Lemma \ref{lem:HandleAttachment}, there is a boundary-critical discrete Morse function $u$ on $\sd^r M$, whose critical interior cells are exactly those of $h$ (namely, one $(d-1)$-face and one  $(d-j)$-face), minus one $(d-1)$-face, plus all of the critical interior $i$-cells of $g$ (which are $c'_{d-i}$). It follows that $u$ has exactly $c'_{d-i} + \delta_{i,j}$ critical interior faces of dimension~$i$. But $c'_{d-i} + \delta_{i,j}= c_{d-i}$.
\end{proof}

\smallskip

\begin{rem} \label{rem:gallaisgap} 
As we mentioned in the introduction, Gallais proved in \cite{Gallais} that every Morse vector is also a discrete Morse vector on some PL triangulation.  
The proof in \cite{Gallais} has a gap, due to the fact that it is unknown whether any simplicial subdivision $X_0$ of the simplex is collapsible \cite[Problem 4.18]{KirbyProblems}. (Specifically, it is not clear whether the triangulation $X$ constructed in \cite[p.~240]{Gallais}, which is a cone over the CW complex $X_0 \cup \Delta_d \cup (\mathbb{I} \times \, \partial \Delta_d)$, collapses down to its bottom $X_0$ as claimed, or not.)

The collapsibility of arbitrary subdivisions of the simplex (and more generally, of any convex polytope) is an old conjecture in PL topology, dating back to Hudson \cite[Sec.~2, p.~44]{Hudson}. There has been however recent progress on this: Adiprasito and the author have shown that if $X_0$ is an arbitrary subdivision of the $d$-simplex, the barycentric subdivision $\sd X_0$ is always collapsible \cite[Main Theorem VI]{AB-MG&C}. This allows to repair Gallais' proof, essentially by subdividing barycentrically the triangulation he considered. 
\end{rem}

\subsubsection*{Caveat: Many barycentric subdivisions might be required}
The previous method suggests that, for every PL triangulation $M$, there exists an integer $r = r(M)$ such that $\sd^r M$ admits perfect discrete Morse functions. This is true for spheres, but not for arbitrary manifolds: Many manifolds cannot admit perfect (smooth or discrete) Morse functions, for example if they have torsion in homology. (See also Corollary \ref{cor:li} below.) Moreover, even for spheres, there is no universal integer $r$ such that, for every PL triangulation $M$, $\sd^r M$ admits perfect discrete Morse function. 

\begin{prp}[Lickorish~\cite{Lickorish}, cf.~also Benedetti--Ziegler~\cite{BZ}]
\label{prp:uglyspheres}
For each integer $r \ge 0$ and for every integer $d \ge 3$, there is a PL $d$-sphere $S$ such that $\sd^r S$ minus any facet is not collapsible.
\end{prp}

In fact, one can even show that for every positive integers $r$, $s$ and $d$, with $d \ge 3$, there is a PL $d$-sphere $S$ such that $\sd^r S$ does not admit any discrete Morse function with fewer that $s$ critical $(d-1)$-faces \cite{B-DMT4MWB}.

In particular, Theorem \ref{thm:PLtoDiscrete} is not algorithmically efficient, as one cannot bound how large $r$ can be. This is consistent with Novikov's famous result that for each $d \ge 5$, there cannot be an algorithm deciding whether a given simplicial complex is a  $d$-sphere or not. (Clearly, given a $d$-sphere $S$ and an integer $k$, there \emph{is} an algorithm to decide if $\sd^k S$ admits $(1, 0, \ldots, 0, 1)$ as discrete Morse vector: we could just try all possible Morse matchings. If it does admit it, then by Forman's sphere theorem $sd^k S$ is homeomorphic to the sphere, and so is $S$.) Using this one can prove the following: for each $d \ge 5$, there is no computable function $f_d : \mathbb{N} \rightarrow \mathbb{N}$ such that, for every PL $d$-sphere 
$S$ with $N$ facets, the $f_d(N)$-th barycentric subdivision of $S$ admits a perfect discrete Morse vector. 

\subsection{An application to 3-manifolds} \label{sec:heegaard}
A \emph{handlebody} is a thickened graph in $\R^3$. A \emph{genus-$g$ Heegaard splitting} of a $3$-manifold $M$ is a decomposition of $M$ into two handlebodies, glued together at a genus-$g$ surface $S_g$. The   \emph{Heegaard genus} of a $3$-manifold is the smallest $g$ for which the manifold has a genus-$g$ Heegaard splitting.

\begin{example}
The solid torus is the result of thickening a cycle; hence it is a handlebody. The complement of a solid torus inside the $3$-sphere is also a solid torus. So, the $3$-sphere has a genus-$1$ Heegaard splitting. It has also a genus-$0$ splitting: The complement of a ball inside the $3$-sphere, is a ball. So the Heegaard genus of the sphere is zero.
\end{example}

Heegaard splittings were introduced around 1900 by Heegaard, and later absorbed into the more general language of handle decompositions. In fact, every genus-$g$ handlebody can be viewed as the result of attaching exactly $g$ handles of index~$1$ to a handle of index~$0$. On the other hand, given a handle decomposition of a closed $3$-manifold with $1$ handle of index~$0$, $g$ handles of index~$1$, $g$ handles of index~$2$, and $1$ handle of index~$3$, this can be viewed as a Heegaard splitting: the $i$-handles with $i\le 1$ form the first handlebody, the $i$-handle with $ i \ge 2$ form the second ``dual'' handlebody. One can easily see that the two handlebodies are homeomorphic.

\begin{theorem} \label{prp:HeegaardMorse}
For a closed connected $3$-manifold, the following are equivalent:
\begin{compactenum}[\rm (i) ]
\item $M$ admits a genus-$g$ Heegaard splitting.
\item $M$ admits a handle decomposition with $c_i$ handles of index~$i$, where $(c_0, c_1, c_2, c_3) = (1,g,g,1)$.
\item $M$ admits a Morse function with $c_i$ critical points of index~$i$, where $(c_0, c_1, c_2, c_3) = (1,g,g,1)$.
\item Some triangulation of $M$ admits a discrete Morse function with $c_i$ critical $i$-faces, where $(c_0, c_1, c_2, c_3) = (1,g,g,1)$.
\end{compactenum}
\end{theorem}

\begin{proof}
For (i) $\Leftrightarrow$ (ii) see above; for (ii) $\Leftrightarrow$ (iii) see Section \ref{sec:MorseHandle}. The new part is the equivalence of (ii) and (iv), which is given by Theorems \ref{thm:dmttopl} and \ref{thm:PLtoDiscrete}, and the fact that all triangulations of $3$-manifolds are PL.
\end{proof}

\begin{cor} \label{cor:heegaard}
 The Heegaard genus $g(M)$ of a closed $3$-manifold $M$ is the smallest $g$ such that some triangulation of the manifold admits a discrete Morse function with discrete Morse vector ${\bf c}=(1,g,g,1)$. 
\end{cor}

This result has two consequences in discrete Morse theory. The first one is of algorithmic nature: Li's recent algorithm to compute the Heegaard genus of a manifold \cite{TaoLiAlgorithm} yields an algorithm to determine the best discrete Morse vector over all possible triangulations. The other way around, the heuristic by the author and Lutz to quickly compute some random discrete Morse vectors \cite{BenedettiLutz} yields a heuristic to obtain some Heegaard splittings quickly, albeit with no guarantee of minimality.

The second consequence of Corollary \ref{cor:heegaard} is of topological nature. The work by Boileau, Zieschang, Schultens, Weidmann and Li has revealed topological obstructions for the existence of Heegaard splittings of low genus. Recall that $\operatorname{rank}(M)$ is the minimal number of elements needed to generate $\pi_1(M)$. It is easy to see that $\operatorname{rank}(M) \le g(M)$. In the Sixties, Waldhausen conjectured that equality would hold for all $3$-manifolds. The conjecture was later disproved:

\begin{theorem}[Boileau--Zieschang \cite{BoileauZieschang}, Schultens--Weidmann \cite{SchultensWeidmann}, Li \cite{TaoLi}]
There are Seifert fibered spaces, graph manifolds, and hyperbolic $3$-manifolds with rank strictly smaller than the Heegaard genus. In fact, the distance between rank and genus can be arbitrarily high.
\end{theorem}

\begin{cor} \label{cor:li} For every $g>0$, there is a $3$-manifold $M$ such that any discrete Morse function on any triangulation of $M$ has more than $g + \operatorname{rank} (M)$ critical edges.  
\end{cor}

On these manifolds, \emph{no matter how much we subdivide the triangulation}, no discrete Morse function will be sharp in reading off the homotopy.

\subsection{Non-PL discrete Morse vectors} \label{sec:karim}
Say that a topological manifold has both PL and non-PL triangulations. Suppose we want to find the smallest discrete Morse vector, namely, a triangulation with a discrete Morse function with the fewest possible number of critical cells. It is not clear whether we should look among PL triangulations, rather than among non-PL ones. Perhaps a simple observation to keep in mind is that cones are always collapsible; hence cones over non-PL balls yield examples of collapsible, non-PL triangulations. 
This indicates that non-PL objects might be ``competitive'' with PL ones, when it comes to having perfect discrete Morse functions. 

The following $5$-dimensional example illustrates a more bizarre possibility.

\begin{example} \label{ex:AdiprasitoBenedetti}
Adiprasito, Lutz and the author recently constructed a collapsible non-PL triangulation of some smooth contractible $5$-manifold \cite{randomDMT2}. This manifold is not homeomorphic to a ball, because the fundamental group of its boundary is the binary icosahedral group. Since the manifold is smooth, it admits PL triangulations by Theorem \ref{thm:pltr}, but none of these PL triangulations is collapsible. In fact, by Whitehead's theorem (Corollary \ref{cor:Whitehead}), all PL collapsible manifolds are homeomorphic to balls. 
\end{example}

Hence in the quest for an optimal discrete Morse function, in some cases it is better to look among \emph{non-PL} triangulations.
We proceed now to construct the first examples of a non-PL triangulation of the $5$-sphere that admits a perfect discrete Morse function. The idea for the following construction comes from Karim Adiprasito.

\begin{prp}\label{prp:forman}
For every $d\geq 5$, there exists a non-PL triangulation $B$ of the $d$-ball such that $B$ is collapsible and $\partial B$ is even shellable. In particular, $B$ is not PL, but its boundary is.
\end{prp}

\begin{proof}
Let $M$ be a PL triangulation of a contractible $(d-1)$-manifold, such that $M$ is not homeomorphic to a ball and the double of $M$ is PL homeomorphic to the standard sphere. Such examples exist: For example, for $d-1=4$, one can take as $M$ an arbitrary triangulation of the Mazur manifold. (For $d-1\neq4$, the double of every contractible PL $(d-1)$-manifold is automatically a sphere by Theorem \ref{thm:poincare}.)

Let $H = \partial M$ be the boundary of $M$. This $H$ is a simplicial PL homology $d$-sphere. Since $M$ is not a ball, by Corollary \ref{cor:pwb} $H$ is not simply connected.  Let $p$ be a point not in $M$, and let $C_0$ denote the manifold obtained by gluing the cone $p \ast H$ to $M$ along their common boundary $H$. $C_0$ is not a manifold, but it is a homology-manifold, and has the same homology as a $d$-sphere. Furthermore $C_0$ is simply connected, since $H$ is connected. Consider the complex $C_1$, obtained as the suspension $\Sigma C_0$ of $C_0$. Clearly, $C_1$ is a homology manifold. Let us examine the links of the vertices of $C_1$. There are three cases to consider:
\begin{compactenum}[ (a) ]
\item If $v$ lies in $M$, then  $\Lk(v,C_0)$ is a sphere, hence $\Lk(v,C_1)=\Sigma\Lk(v,C_0) $ is a sphere as well. In particular, $\Lk(v,C_1)$ is simply connected.
\item If $v = p$ (the apex of the cone), $\Lk(v,C_1)$ is the suspension of $\Lk(v,C_0)$, which is equivalent to $H$ and hence connected. By Seifert--Van Kampen, $\Lk(v,C_1)$ is simply connected (although it is not a sphere).
\item If $v$ is one of the apices of the suspension $\Sigma C_0$, then $\Lk(v,C_1) = C_0$. In particular, although it is not a manifold, it is simply connected.
\end{compactenum}

Hence in all cases $\Lk(v,C_1)$ is simply connected. By the Edwards--Cannon criterion (Theorem \ref{thm:edwards}), $C_1$ is a manifold. It is not PL, though. In fact, if $a$ and $b$ are the apices of the suspension $\Sigma C_0$, then the PL-singular set $S$ of $C_1$ consists of the vertices $a$, $ b$ and $p$, and of the edges connecting $a$ to $p$ and $b$ to $p$.

Let us consider the submanifold $B$ of $\sd^2 C_1$ obtained as the union of all faces of $\sd^2 C_1$ intersecting $\sd^2 S$. Then $B$ collapses to $S$ (cf.\ \cite[Lemma III.9 and III.10]{Glaser}), and since $S$ is a tree, and in particular collapsible, so is $B$. Furthermore, the boundary of $B$ is PL homeomorphic to the double of the contractible manifold $M$, which is a PL sphere by assumption. Indeed, $\partial B$ is PL homeomorphic to the gluing of $\Lk(a,C_1)-p\cong M$ and $\Lk(b,C_1)-p\cong M$ along their common boundary $\Lk(p,C_1)-\{a\} - \{b\}\cong H$. 
In particular, by Corollary \ref{cor:pwb}, $B$ is a ball. However, $B$ is not a PL ball, since its PL singular set is nonempty. In conclusion, $B$ is a collapsible, non-PL ball, whose boundary is a PL sphere. By Lemma \ref{lem:ShellableSubdivision}, every PL sphere becomes shellable after sufficiently many iterated barycentric subdivisions. So, up to replacing $B$ with $\sd^r B$ (for a suitably large $r$), we can also assume that $\partial B$ is shellable. 
\end{proof}

\begin{cor} \label{cor:A}
For each $d \ge 5$, some non-PL $d$-spheres admit perfect discrete Morse functions.
\end{cor}

\begin{proof} Consider the non-PL collapsible $d$-ball $B$ with PL shellable boundary constructed in Proposition \ref{prp:forman}. Let $y$ be a vertex not in $B$, and let 
\[S \; = \; \partial (y \ast B) \; = \; B \, \cup \, (y \ast \, \partial B).\]
Since $\partial B$ is shellable, so is the ball $B' = (y \ast \, \partial B)$. Let $\Delta'$ be any facet of $B'$. Since $B'$ is shellable, $B'$ is also endocollapsible, so $B' - \Delta'$ collapses onto $\partial B = \partial B'$. The same collapsing sequence shows that $S - \Delta'$ collapses onto $B$. But $B$ is collapsible; so $S - \Delta'$ is collapsible.
\end{proof}

\begin{rem} \label{rem:WhiteheadStrengthened}
If $C$ is a collapsible complex, then $\sd C$ collapses to a facet of $\sd C$. Thus, we can modify the construction of Corollary \ref{cor:A} to obtain a non-PL $5$-sphere $S$ such that $S - \Delta'$ collapses onto another facet $\Delta''$. Let $A$ be the $5$-ball $S - \Delta''$. Then, with the same collapsing sequence, $A - \Delta'$ collapses onto $\partial \Delta'' = \partial A$. So $A$ yields an example of an endo-collapsible ball that is not PL. 
This proves that Theorem \ref{thm:BallTheorem} (``every endo-collapsible manifold is a ball'') is more general than Whitehead's result (Corollary \ref{cor:Whitehead}) that ``every collapsible PL manifold is a ball''.
\end{rem}

\subsection{LC triangulations and geometric connectivity} \label{sec:LC}
In this section we apply our previous results to show that every simply connected smooth $d$-manifold, with $d \ne 4$,  admits LC triangulations; in contrast, no triangulation of Mazur's simply connected $4$-manifold is LC.

By a {\em tree of $d$-simplices} we mean a triangulation of a $d$-ball whose dual graph is a tree. {\em Locally constructible} (shortly, LC) triangulations are those obtainable from some tree of $d$-simplices, with $d \ge 2$, by repeatedly identifying two adjacent $(d-1)$-simplices in the boundary \cite{BZ}. LC triangulations were introduced by Durhuus--Jonsson \cite{DJ}, who proved  the following result:

\begin{theorem}[Durhuus--Jonsson \cite{DJ}] \label{thm:DJ}
Every $3$-manifold admitting LC triangulations is homeomorphic to the $3$-sphere.
In higher dimensions, every manifold admitting LC triangulations is simply connected. 
\end{theorem}

Durhuus and Jonsson conjectured that every triangulation of the $3$-sphere is LC. They also wondered if perhaps every triangulation of every simply connected manifold is LC. As they noticed, a positive solution of both conjectures would have immediately implied (via Theorem \ref{thm:DJ}) the $3$-dimensional Poincar\'e conjecture, at the time still open. Perelman's work sparkled some optimism for Durhuus--Jonsson's conjectures but as it later turned out, not all triangulations of the $3$-sphere are LC.

\begin{theorem}[Benedetti--Ziegler \cite{BZ}] \label{thm:BZ} Let $d \ge 2$ be an integer. Let $S$ be an arbitrary triangulation of the $d$-sphere. $S$ is LC if and only if $S$ admits a discrete Morse function without critical $(d-1)$-faces. In particular, one can construct non-LC triangulations of the $d$-sphere for each $d \ge 3$.
\end{theorem}

This was later extended in \cite{B-DMT4MWB} to manifolds with boundary, as follows. Let $M$ be an arbitrary  triangulated manifold with boundary. 
The \emph{collapse depth} $\cd M$ of $M$ is the maximal integer $k$ for which there exists a boundary-critical discrete Morse function on $M$ with one critical $d$-cell and no critical interior $(d-i)$-cells, for each $i \in \{1, \ldots, k - 1\}$. Obviously $1 \le \cd M \le \dim M$. We stress that the collapse depth is not a topological invariant: It depends on the triangulation chosen. 

\begin{theorem}[\cite{B-DMT4MWB}] \label{thm:cdepth2}
Let $M$ be a triangulation of a manifold with boundary. $M$ is LC if and only if $\cd M \ge 2$.
\end{theorem}

Theorem \ref{thm:BZ} was part of the author's PhD thesis. During the doctoral defense, \v{Z}ivaljevi\'c made the following insightful conjecture:

\begin{conjecture}[\Ziv $\!\!$]
Every simply connected smooth $d$-manifold admits \emph{some} LC triangulations. 
\end{conjecture}

To tackle the conjecture, let us recall a classical result in PL topology. Given an integer $p$ in $\{1, \ldots, d\}$, a 
\emph{geometrically $p$-connected} manifold is a manifold that admits a PL handle decomposition with one $0$-handle and no further handles of index $\le p$ \cite{Wall}. It is easy to see that geometri\-cally $p$-connected implies $p$-connected. In the Sixties, Wall discovered a surprising converse implication: 

\begin{theorem}[Wall~\cite{Wall}] \label{thm:wall}
Let $M$ be a $p$-connected PL $d$-manifold, with or without boundary. If $p \le d-4$, then $M$ is also geometrically $p$-connected.
\end{theorem}

The bound $p \le d-4$ is best possible: In fact, Mazur described a contrac\-tible $4$-manifold different than the $4$-ball \cite{Mazur}, and Casson later proved that every PL handle decomposition of Mazur's manifold must contain $1$-handles, cf.~\cite[Problem 4.18]{KirbyProblems}. Hence Mazur's $4$-manifold is $1$-connected but not geometrically. However, every $(d-3)$-connected $d$-manifold becomes geometrically $(d-3)$-connected after sufficiently many ``stabilizations'', cf.~Quinn~\cite[Theorem~1.2]{Quinn}. 
Our next result shows that geometric connectivity and collapse depth (of a suitable subdivision) are essentially the same invariant, up to a shift of one:

\begin{prp} \label{prp:wall2}
Let $M$ be any PL triangulation of a given manifold.
\begin{compactenum}[\rm (1)]
\item If $M$ is 
geometrically $p$-connected, for $r$ large enough $\cd \sd^r M \ge p+1$.
\item If $\,\cd M \ge p + 1$, then $M$ is geometrically $p$-connected.
\end{compactenum}
\end{prp}

\begin{proof}
\begin{compactenum}[(1)]
\item By definition, $M$ has a PL handle vector of the form $(1, 0, \ldots, 0, c_{p+1}, \ldots, c_{d})$. By Theorem \ref{thm:PLtoDiscrete}, some iterated barycentric subdivision of $M$ admits $(1, 0, \ldots, 0, c_{p+1}, \ldots, c_{d})$ as discrete Morse vector. 
By Corollary \ref{cor:dualityMe}, $(c_d, \ldots, c_{p+1}, 0, \ldots, 0, 1)$ is an interior discrete Morse vector on some iterated barycentric subdivision $M'$ of $M$. Hence the collapse depth of $M'$ is at least $p+1$.
\item If $\cd M \ge p + 1$, we can find on $M$ a boundary-critical discrete Morse function with one critical $d$-cell and no critical interior $(d-i)$-cells, for each $i \in \{1, \ldots, p\}$. By  Corollary \ref{cor:dualityMe}, for some $r$ the subdivision $\sd^r M$ admits a discrete Morse function without critical $i$-faces, for $1 \le i \le p$.  
By Theorem \ref{thm:dmttopl},  $M$ is geometrically $p$-connected. \qedhere
\end{compactenum}
\end{proof}

\begin{theorem}\label{thm:zv} Let $d \ge 2$. Every PL triangulation of every simply connected PL $d$-manifold, $d \ne 4$, becomes LC after subdividing it barycentrically suitably many times. In contrast, the Mazur $4$-manifold admits no LC triangulation.
\end{theorem}

\begin{proof} 
\begin{compactitem}[--]
\item For $d=2$ the claim is trivial, since every triangulated $2$-ball or $2$-sphere is shellable, and in particular PL, endocollapsible, and LC.
\item For $d=3$, let $M$ be a simply connected $3$-manifold. If $\partial M = \emptyset$, by Perelman's solution of the Poincar\'e conjecture $M$ is a $3$-sphere. By a result of Adiprasito and the author, for any PL triangulation of the $3$-sphere there exists an integer $r$ such that $\sd^r S$ is shellable \cite{AB-SSZ}, hence LC. If $\partial M \ne \emptyset$, let $b$ be the number of connected components of its boundary. Each of these connected components is a simply connected closed $2$-manifold, that is, a $2$-sphere. Via the Poincar\'{e} conjecture, we conclude that $M$ is topologically a ``piece of Swiss cheese with $b-1$ holes'' (that is, the result of removing $b$ disjoint $3$-balls from a $3$-sphere). 

We proceed by induction on $b$. If $b=1$, then $M$ is a PL $3$-ball, thus it has a shellable subdivision by Adiprasito--Benedetti \cite{AB-SSZ} and we are done. If $b \ge 2$, 
 up to replacing the triangulation with an iterated barycentric subdivision of it, we can find an embedded annulus $A$ inside $M$ around one of the holes such that $M$ splits as $B \cup M_1$, where $B$ is a $3$-ball, $B \cap M_1 = A$,  and $M_1$ is a ``piece of Swiss cheese with $b-2$ holes''. By the inductive assumption, up to taking further barycentric subdivisions we can assume that the triangulation of $M_1$ is LC and that of $B$ is shellable (so in particular LC). 
Moreover, any triangulation of $A$ has connected dual graph. Using \cite[Lemma~2.23]{BZ}, we conclude that the resulting subdivision of $M$ is LC.  
\item For $d=4$, Mazur's manifold is not geometrically $1$-connected, and any triangulation of it (being $4$-dimensional) is PL. By Proposition \ref{prp:wall2},(2), any triangulation of Mazur's manifold has collapse depth $\le 1$, hence it cannot be LC by Theorem \ref{thm:cdepth2}.
\item Finally, if $d \ge 5$ every simply connected $d$-manifold is $1$-connected by Theorem \ref{thm:wall}. By Proposition \ref{prp:wall2}, (1), for any PL triangulation $M$ we can find an integer $r$ such that \[\cd \sd^r M \ge 2.\] Via Theorem \ref{thm:cdepth2}, we conclude that $\sd^r M$ is LC. \qedhere
\end{compactitem}
\end{proof}


\v{Z}ivaljevi\'c 's conjecture remains open for closed $4$-manifolds: are there simply connected closed smooth $4$-manifolds that do not admit LC triangulations? Via Theorem \ref{thm:cdepth2} and Proposition \ref{prp:wall2}, this question is equivalent to the following, long-standing open problem:

\begin{problem}[{cf.~Kirby \cite[Problem 4.18]{KirbyProblems}}]
Is every simply-connected closed manifold also geometrically $1$-connected?
\end{problem}

As a remark, we mention that for simply connected smooth manifolds of higher dimension, much more on the Morse vector is known:

\begin{theorem}[Sharko~{\cite[pp.~27--28]{Sharko}}] \label{thm:sharko}
Every contractible smooth $d$-manifold $M$, with $d\ge 6$, admits a handle decomposition with exactly one  $0$-handle, $m$ handles of index $d-3$ and $m$ handles of index $d-2$, where $m$ is the minimal number of generators of the relative homotopy group $\pi_2 (M, \partial M)$. (When $M$ is a ball, $m=0$.)
\end{theorem}

This of course can be translated in a statement on the discrete Morse vector via Theorem \ref{thm:EquivalenceOfMorseVectors}. Understanding (optimal) smooth/discrete Morse vectors for manifolds that are \emph{not} simply connected, instead, seems to be a much more difficult problem. For a survey of what has been achieved so far, we refer the reader to Sharko~\cite[Chapter~7]{Sharko}.

\subsection*{Acknowledgments} 
Special thanks to Karim Adiprasito, who was of great help and allowed me to include some results of joint work (like Theorem \ref{thm:dmttopl} or Proposition \ref{prp:forman}). A very early version of this work emerged in 2010 after a question by Rade \Ziv and an inspiring conversation with Rob Kirby; to both of them the author is indebted. Thanks also to Alex Suciu, John Shareshian, Marty Scharlemann, Elmar Vogt, G\"{u}nter Ziegler, Lou Billera, Ed Swartz, Anders Bj\"orner, Ne\v{z}a Mramor, Ulrich Bauer, and Jonathan Barmak, for stimulating conversations and helpful references. Finally, thanks to Liviu Nicolaescu for pointing out the reference \cite{Gallais}, to Jos\'e Samper for correcting typos, and to the anonymous referee for improving the presentation.

\bibliographystyle{amsalpha}

\small
\begin{thebibliography}{99}
\itemsep=-1.4mm
\bibitem[Adi13]{Adiprasito} K.~A.~Adiprasito, \emph{Methods from Differential Geometry in Polytope Theory}. PhD Thesis, Berlin, May 2013. Available at
\url{arxiv:1403.2657}.

\bibitem[AB11]{AB-MG&C}
K.~A.~Adiprasito and B.~Benedetti, \emph{Metric geometry, convexity and collapsibility},
  preprint, \url{arXiv:1107.5789}.

\bibitem[AB12]{AB-SSZ}
K.~A.~Adiprasito and B.~Benedetti, \emph{Subdivisions, shellability and collapsibility of products}, Combinatorica, to appear. Preprint at
  \url{arXiv:1202.6606}.
  
\bibitem[ABL14]{randomDMT2}
  K.~A.~Adiprasito, B.~Benedetti, F.~H.~Lutz, \emph{Extremal examples of collapsible complexes and random discrete Morse theory},
  preprint, \url{arXiv:1404.4239}.

\bibitem[Akb10]{akb}
S.~Akbulut, \emph{Cappell-{S}haneson homotopy spheres are standard}, Ann. of
  Math. (2) \textbf{171} (2010), no.~3, 2171--2175.

\bibitem[BHHM]{BHHM}
M.~Behrens, M.~A.~Hill, M.~J.~Hopkins and M.~Mahowald, \emph{Exotic spheres detected by topological modular forms}, in preparation (2014). Cf.\ also the research report by M. Behrens available at \url{www3.nd.edu/~mbehren1/presentations/Exotic_spheres_midwest.pdf} (retrieved Dec.3, 2014). 

\bibitem[Ben10]{B-JCTA}
B.~Benedetti, \emph{Collapses, products and {LC} manifolds}, J. Comb. Theory, Ser. A (2010), 586--590.


\bibitem[Ben12]{B-DMT4MWB}
B.~Benedetti, \emph{Discrete Morse theory for manifolds with boundary}, Trans. Amer.
  Math. Soc. \textbf{364} (2012), 6631--6670.

\bibitem[BL14]{BenedettiLutz}
B.~Benedetti and F.~H.~Lutz, \emph{Random discrete Morse theory and a new library of triangulations},  Experimental Mathematics, Vol. 23, Issue 1 (2014), 66--94. 
  
\bibitem[BZ11]{BZ}
B.~Benedetti and G.~M. Ziegler, \emph{On locally constructible spheres and
  balls}, Acta Math. \textbf{206} (2011), no.~2, 205--243.

\bibitem[BL00]{BjornerLutz}
A.~Bj\"orner and F.~H. Lutz, \emph{Simplicial manifolds, bistellar flips and a 16-vertex triangulation of the Poincaré homology 3-sphere}, Exp. Math. 9 (2000), 275--289.
  
\bibitem[BZ84]{BoileauZieschang}
M.~Boileau and H.~Zieschang, \emph{Heegaard genus of closed orientable Seifert $3$-manifolds}, Inventiones Math. 76 (1984), 455--468.
  
  
\bibitem[Buo03]{Buoncristiano}
S.~Buoncristiano, \emph{Fragments of geometric topology from the
  {S}ixties}, vol.~6 of Geometry \& Topology Monographs, Coventry (2003).

\bibitem[Cai61]{Cairns}
S.~S. Cairns, \emph{A simple triangulation method for smooth manifolds}, Bull.
  Amer. Math. Soc. \textbf{67} (1961), 389--390.

\bibitem[Can79]{Cannon}
J.~W. Cannon, \emph{Shrinking cell-like decompositions of manifolds. Codimension three}, Annals of Math. (2), vol. 110 (1979), no. 1, pp. 83-–112.

\bibitem[Cer68]{Cerf}
J.~Cerf, \emph{Sur les diff\'eomorphismes de la sph\`ere de dimension trois
  {$(\Gamma _{4}=0)$}}, Lecture Notes in Mathematics, No. 53, Springer-Verlag,
  Berlin, 1968.

\bibitem[Cer70]{CerfPiso}
J.~Cerf, \emph{La stratification naturelle des espaces de fonctions
  diff\'erentiables r\'eelles et le th\'eor\`eme de la pseudo-isotopie}, Inst.
  Hautes \'Etudes Sci. Publ. Math. (1970), no.~39, 5--173. 

\bibitem[DJ95]{DJ}
B.~Durhuus and T.~Jonsson, \emph{Remarks on the entropy of 3-manifolds}.
  Nucl.~Phys.~B 445 (1995), 182--192. 

\bibitem[Edw78]{Edwardsicm}
R.~D. Edwards, \emph{The topology of manifolds and cell-like maps}, Proc. ICM 1978, Helsinki, 111--127.

\bibitem[Edw06]{EdwardsPreprint}
R.~D. Edwards, \emph{Suspensions of homology spheres}, preprint (publ. 2006, but written around 1978) at \url{arXiv:math/0610573}.

\bibitem[For02]{FormanUSER}
R.~Forman, \emph{A user's guide to {D}iscrete {M}orse Theory}, Sem. Lothar.~Comb. 48:~Art B48c (2002), 35pp.
 
 
\bibitem[Fre82]{Freedman}
M.~Freedman, \emph{The topology of four-dimensional manifolds}, J.~Diff.~Geom. 17 (1982), 357--453.

\bibitem[FGMW10]{SPC}
M.~Freedman, R.~Gompf, S.~Morrison, and K.~Walker, \emph{Man and machine
  thinking about the smooth 4-dimensional {P}oincar\'e conjecture}, Quantum
  Topol. \textbf{1} (2010), no.~2, 171--208.

\bibitem[GS80]{GalewskiStern}
D.~E. Galewski and R.~J. Stern, \emph{Classification of simplicial
  triangulations of topological manifolds}, Ann. of Math. (2) \textbf{111}
  (1980), no.~1, 1--34.

\bibitem[Gal10]{Gallais}
E.~Gallais, \emph{Combinatorial realization of the Thom--Smale complex via discrete Morse theory}, Ann. Sc. Norm. Super. Pisa, Classe di Scienze (5) 9, No. 2 (2010), 229--252.

  
\bibitem[Gla70]{Glaser}
L.~C. Glaser, \emph{Geometrical combinatorial topology I}, Van Nostrand
  Reinhold, New York, 1970.

\bibitem[Gom10]{Gompf}
R.~E. Gompf, \emph{More {C}appell-{S}haneson spheres are standard}, Algebr.
  Geom. Topol. \textbf{10} (2010), no.~3, 1665--1681.

\bibitem[Goo68]{Goodrick}
R.~E. Goodrick, \emph{Non-simplicially collapsible triangulations of
  {$I^{n}$}}, Proc. Cambridge Philos. Soc. \textbf{64} (1968), 31--36.

\bibitem[Hat01]{Hatcher}
A.~Hatcher, \emph{Algebraic Topology}. Cambrige University Press, 2001.

\bibitem[HM74]{HundM}
M.~W. Hirsch and B.~Mazur, \emph{Smoothings of piecewise linear manifolds},
  Princeton University Press, Princeton, NJ, 1974, Annals of Mathematics
  Studies, No. 80. 

\bibitem[HS70]{HsiangShaneson}
W.~Hsiang and J.~L. Shaneson, \emph{Fake tori}, Topology of {M}anifolds
  ({P}roc. {I}nst., {U}niv. of {G}eorgia, {A}thens, {G}a., 1969), Markham,
  Chicago, Ill., 1970, pp.~18--51.

\bibitem[Hud69]{Hudson}
J.~F.~P.~Hudson,
\emph{Piecewise Linear Topology}, University of Chicago Lecture Notes, W. A. Benjamin,
Inc., New York-Amsterdam, 1969.

\bibitem[JM09]{JerseMramorKosta}
G.~Jer\u{s}e and N.~{Mramor Kosta}, \emph{Ascending and descending regions  of a discrete {M}orse function}, Comput. Geom. 42 (2009), 639--651.

  
  
\bibitem[Ker60]{Kervaire}
M.~A. Kervaire, \emph{A manifold which does not admit any differentiable
  structure}, Comment. Math. Helv. \textbf{34} (1960), 257--270.




\bibitem[KM63]{MilnorKervaire}
M.~A. Kervaire and J.~W. Milnor, \emph{Groups of homotopy spheres. {I}}, Ann.
  of Math. (2) \textbf{77} (1963), 504--537.
  
 
\bibitem[KS77]{KirbyProblems}
R.~C. Kirby, \emph{Problems in low-dimensional topology}. In Geometric Topology, W. Kanez ed., AMS, 1997, 35--473. An updated version is maintained at the website
\url{math.berkeley.edu/~kirby/}.


\bibitem[KS77]{KirSie}
R.~C. Kirby and L.~C. Siebenmann, \emph{Foundational essays on topological
  manifolds, smoothings, and triangulations}, Princeton Univ. Press,
  Princeton, 1977. 

\bibitem[Lee03]{Lee}
J.~M.~Lee, \emph{Introduction to smooth manifolds}. Graduate Texts in Mathematics 218. Springer-Verlag, New York, 2003.

\bibitem[Li11]{TaoLiAlgorithm}
T.~Li, \emph{An algorithm to determine the {H}eegaard genus of a 3-manifold},
  Geom. Topol. \textbf{15} (2011), no.~2, 1029--1106.

\bibitem[Li13]{TaoLi}
T.~Li, \emph{Rank and genus of 3-manifolds}, J. Amer. Math. Soc. 26 (2013), 777--829.

\bibitem[Lic91]{Lickorish}
W.~B.~R. Lickorish, \emph{Unshellable
  triangulations of spheres}, Europ. J. Combin., 12 (1991), 527--530.



\bibitem[Man13]{Manolescu}
C. Manolescu, \emph{Pin(2)-equivariant Seiberg-Witten Floer homology and the Triangulation Conjecture}, preprint (2013) available at \url{www.math.ucla.edu/~cm/papers.html}


\bibitem[Mat02]{Matsumoto}
Y.~Matsumoto, \emph{An introduction to {M}orse theory}, Translations of
  Mathematical Monographs, vol. 208, American Mathematical Society, Providence,
  RI, 2002.
  
\bibitem[Maz61]{Mazur}
 B.~Mazur, \emph{A note on some contractible $4$-manifolds}. Ann.~of Math.
  (2) 73 (1961), 221--228.


\bibitem[Mil09]{MilnorDiffTop}
J.~Milnor, \emph{Differential topology forty-six years later}, in Notices of the AMS, vol.~58 (2011) 804--809.

\bibitem[Mil09]{MilnorFifty}
J.~Milnor, \emph{Fifty years ago: topology of manifolds in the 50's and 60's}, in Low Dimensional Topology, vol.~15 of {IAS / P}ark City Mathematics Series, AMS (2009), 7--20.

\bibitem[Mil63]{Milnor}
J.~Milnor, \emph{Morse theory}, Based on lecture notes by M. Spivak and R. Wells.
  Annals of Mathematics Studies, No. 51, Princeton Univ. Press, Princeton, 1963.

\bibitem[Mil56]{Milnorexotic}
J.~Milnor, \emph{On manifolds homeomorphic to the {$7$}-sphere}, Ann. of Math.
  (2) \textbf{64} (1956), 399--405.

\bibitem[Moi52]{Moise}
E.~E. Moise, \emph{Affine structures in {$3$}-manifolds. {V}. {T}he
  triangulation theorem and {H}auptvermutung}, Ann. of Math. (2) \textbf{56}
  (1952), 96--114.

\bibitem[Mor25]{Morse}
M.~Morse, \emph{Relations between the critical points of a real function
  of $n$ independent variables}, Trans. Amer. Math. Soc. 27 (1925),
  345--396.

\bibitem[Mun84]{Munkres}
J.~R. Munkres, \emph{Elements of algebraic topology}, Addison--Wesley, Menlo Park (1984). 
  
  
  \bibitem[Mun60]{Munkressmooth}
J.~R. Munkres, \emph{Obstructions to the smoothing of piecewise-differentiable
  homeomorphisms}, Ann. of Math. (2) \textbf{72} (1960), 521--554. 

\bibitem[Pac87]{Pachner0}
U.~Pachner, \emph{Konstruktionsmethoden und das kombinatorische
  {H}om\"{o}omorphieproblem f\"{u}r {T}riangulationen kompakter semilinearer
  {M}annigfaltigkeiten}, Abh. Math. Sem. Hamb. 57 (1987), 69--86.

\bibitem[Per02]{Perelman}
G.~Perelman, \emph{The entropy formula for the {R}icci flow and its
  geometric applications}, preprint (2002) at \url{arxiv:math/0211159}.


\bibitem[Qui83]{Quinn}
F.~Quinn, \emph{The stable topology of $4$-manifolds}. Topology and its Applications 15 (1983), 71--77.


\bibitem[Ran96]{RanickiHV}
A.~A. Ranicki, \emph{On the {H}auptvermutung}, The {H}auptvermutung book,
  $K$-Monogr. Math., vol.~1, Kluwer Acad. Publ., Dordrecht, 1996, pp.~3--31.


\bibitem[Rat06]{Ratcliffe}
J.~G. Ratcliffe, \emph{Foundations of hyperbolic manifolds}, Springer Verlag, 2006.
  
  
\bibitem[Rie53]{Riemann}
B.~Riemann, \emph{Gesammelte mathematische {W}erke und wissenschaftlicher
  {N}achlass}, Dover Publications Inc., New York, 1953. 

\bibitem[Rok52]{Rokhlin}
V.~A. Rokhlin, \emph{New results in the theory of four-dimensional manifolds},
  Doklady Akad. Nauk SSSR (N.S.) \textbf{84} (1952), 221--224.

\bibitem[RS72]{RourkeSanderson}
C.~P. Rourke and B.~J. Sanderson, \emph{Introduction to piecewise-linear
  topology}, Springer-Verlag, New York, 1972. 

\bibitem[SW07]{SchultensWeidmann}
J.~Schultens and R.~Weidmann, \emph{On the geometric and the algebraic rank of graph $3$-manifolds}, Pacific J.~Math. 231 (2007), 481--510.

  
\bibitem[Sha93]{Sharko}
V.~V. Sharko, \emph{Functions on manifolds: Algebraic and topological aspects}, Transcr. Math. Monographs 131, AMS, 1993.

\bibitem[Sma62]{Smale}
S.~Smale, \emph{On the structure of manifolds}, Amer. J. Math. 84 (1962),
  387--399.

\bibitem[Sta62]{Stallings}
J.~Stallings, \emph{The piecewise-linear structure of Euclidean space}. Proc. Cambridge Philos. Soc. 58 (1962), 481–-488.

\bibitem[Wal71]{Wall}
C.~T.~C. Wall, \emph{Geometrical connectivity I \& II}. J.~London Math.~Soc. 3 (1971), 597--608.

\bibitem[Whi61]{WhiteheadC1}
J.~H.~C. Whitehead, \emph{Manifolds with transverse fields in {E}uclidean space},
  Ann. of Math. (2) \textbf{73} (1961), 154--212. 

\bibitem[Zee61]{ZeemanPoincare}
E.~C. Zeeman, \emph{The generalised Poincar\'e conjecture}. Bull. Amer. Math. Soc. 67,  270.

\bibitem[Zee66]{ZeemanBK}
E.~C. Zeeman, \emph{Seminar on combinatorial topology.}, Institut des Hautes
  \'Etudes Scientifiques 1963. Chap. 1-6; Chap. 7, (1965); Chap. 8, Paris,
  1966.


\end{thebibliography}

\end{document}